\documentclass[12pt]{amsart}

\usepackage[dvips, dvipdfmx]{graphicx}
\usepackage[abs]{overpic}
\usepackage{amssymb}
\usepackage{amsmath}
\usepackage{amsthm}
\usepackage{mathrsfs}
\usepackage{fancybox}
\usepackage{comment}
\usepackage{layout}
\usepackage{color}
\usepackage{bm}
\usepackage{tabularray}

\theoremstyle{plain}
\newtheorem{theorem}{Theorem}[section]
\newtheorem{lemma}[theorem]{Lemma}

\newtheorem{proposition}[theorem]{Proposition}

\theoremstyle{definition}
\newtheorem{definition}[theorem]{Definition}

\newtheorem{remark}[theorem]{Remark}

\graphicspath{{figures_arXiv/}}

\title[Homotopy theories of colored links and spatial graphs]{Homotopy theories of colored links and spatial graphs}
\author{Yuka Kotorii}
\address[Kotorii]{Mathematics Program, Graduate School of Advanced Science and Engineering, Hiroshima University, 1-7-1 Kagamiyama Higashi-hiroshima City, Hiroshima 739-8521 Japan}
\address[Kotorii]{International Institute for Sustainability with Knotted Chiral Meta Matter (WPI-SKCM$^2$), Hiroshima University, 1-3-2 Kagamiyama Higashi-hiroshima City, Hiroshima 739-8521 Japan}
\email{kotorii@hiroshima-u.ac.jp}

\author[Mizusawa]{Atsuhiko Mizusawa}
\address[Mizusawa]{Faculty of Science and Engineering, Waseda University, 3-4-1 Okubo, Shinjuku-ku, Tokyo 169-8555, Japan} 
\email{a\_mizusawa@aoni.waseda.jp}

\date{March 2023}
\begin{document}

\maketitle


\begin{abstract}
Two links are called link-homotopic if they are transformed to each other by a sequence of self-crossing changes and ambient isotopies. 
The notion of link-homotopy is generalized to spatial graphs and it is called component-homotopy. 
The link-homotopy classes were classified by Habegger and Lin through the classification of the link-homotopy classes of string links. 
In this paper, we classify colored string links up to colored link-homotopy by using the Habegger-Lin theory. Moreover, we classify colored links and spatial graphs up to colored link-homotopy and component-homotopy respectively.  
\end{abstract}

\section{Introduction} \label{intro}

In this paper, we work in the piecewise linear category and components of links and edges of spatial graphs are ordered and oriented. A string link is a proper embedding of closed finite intervals to a cylinder, see Section \ref{ColStringLink} for details. 
\par
Two links (resp. string links) are \textit{link-homotopic} if one can be transformed to the other by a sequence of self-crossing changes (i.e. crossing changes between arcs of the same components) and ambient isotopies (resp. ambient isotopies relative to endpoints of each component). Let $\mathcal{L}(n)$ (resp. $\mathcal{H}(n)$) be the set of the link-homotopy classes of $n$-component links (resp. string links). 
\par
The notion of the link-homotopy was introduced by Milnor and he classified $\mathcal{L}(n)$ for $n=2,3$ by numerical invariants $\overline{\mu}_L(I)$ called Milnor invariants for links $L$ \cite{Mil, Mil2}. Levine \cite{Le2} gave a classification of $\mathcal{L}(4)$ by giving its explicit presentation. Habegger and Lin \cite{HL} gave a classification of $\mathcal{L}(n)$ for all $n$. 
They first clarified the structure of $\mathcal{H}(n)$ (see Theorem \ref{SLclassify} in Section \ref{ColStringLink}) and classified $\mathcal{H}(n)$ by $\nu(n)$ kinds of the Milnor invariants $\mu_{S}(I)$ which are defined for string links $S$, where 
$$\nu(n)\!=\!\sum_{k=2}^{n}(k\!-\!2)!
\left(\!\!
\begin{array}{c}
n \\ k
\end{array}
\!\!\right).
$$
The ranges of $\mu_{S}(I)$ are all integers $\mathbb{Z}$ and $\mathcal{H}(n)$ has a presentation $\mathbb{Z}^{\nu(n)}$ as a set. 
Then Habegger and Lin made a Markov-type theorem for $\mathcal{L}(n)$ by considering the closures of elements of $\mathcal{H}(n)$ 
(see Theorem \ref{Lclassify} in Section \ref{ColStringLink}). They determined the equivalence relation $\sim$ between two string links in $\mathcal{H}(n)$
which become link-homotopic links by taking closures. Therefore $\mathcal{L}(n)$ has presentations  
$$\mathcal{L}(n)=\mathcal{H}(n)/\sim=\mathbb{Z}^{\nu(n)}/\sim'$$
as a set, where $\sim'$ is an equivalence relation corresponding to $\sim$.
Note that, for $n\geq 4$, $L(n)$ has not been classified by numerical invariants. 
However, in the same paper,  Habegger and Lin showed that there is an algorithm which determines whether given two links are link-homotopic or not. This algorithm uses actions of the relation $\sim$ or $\sim'$. 
For $n=4, 5$, the actions of the relation $\sim'$ were calculated explicitly in \cite{Gra, KM2, KM3} and the algorithm can
run, see \cite{KM2, KM3} for details. We remark that, in \cite{Le2, KM2}, several subsets of $\mathcal{L}(4)$ which are classified by numerical invariants are given. 
\\ 

\par
In this paper, we study generalizations of the link-homotopy for colored links and spatial graphs. 

\par
For a link $L$, we make a surjective map $\phi:C(L)\to Col(m)$ from the set $C(L)$ of components of $L$ to the set $Col(m)=\{1,\dots,m\}$ of natural numbers with the cardinality $m$, where we call the map $\phi$ a \textit{coloring} of $L$ and the elements of $Col(m)$ \textit{colors}. A link with a coloring is called a \textit{colored link}. The components mapped to $i$ by the coloring are called \textit{$i$-colored components}. 
For a color $i$, an \textit{$i$-sublink} $L_i$ of a colored link $L$ is a sublink consisting of the components colored by $i$, where $L_i$ inherits the order of $L$. Let $l_i$ be the number of components of $L_i$. 
Without loss of generality, we assume that the $j$-th component of $L_i$ is the $(\sum_{k=1}^{i-1}l_k+j)$-th component of $L$. 
We call the $j$-th component of $L_i$ the $(i,j)$-th component of $L$. Here the order of the component $L$ is equivalent to the lexicographical order of $(i,j)$. 
For simplicity, we denote the pair $(i,l_i)$ of a color $i$ and the number $l_i$ of components of $L_i$
by $i_{l_i}$. 
The sequence ${\bm l}=(1_{l_1},\dots,m_{l_m})$ of the pairs for a colored link $L$ is called a \textit{component decomposition} of $L$. 
We also call a string link with coloring a \textit{colored string link}. For a color $i$, an \textit{$i$-sub-string link} $\sigma_i$ of a colored string link $\sigma$ is a sub-string link consisting of the components with the color $i$.
\par
Two colored links (resp. colored string links) are \textit{CL-homotopic} if one can be transformed to the other by a sequence of crossing changes between components with the same colors and ambient isotopies. 

\par
A \textit{spatial graph} is an image $\psi(G)$ of a spatial embedding $\psi:G\to S^3$ of a finite graph $G$ to the three sphere $S^3$. In this paper, valencies of vertices of a graph are more than or equal to one. 
For the spatial graphs, we give orders to the components and the edges of each component. We call the $j$-th edge of the $i$-th component of the spatial graph its $(i,j)$-th edge. 
Then the lexicographical order of $(i,j)$ gives an order to the set of edges of the spatial graph. 
Two spatial graphs $\psi_1(G)$ and $\psi_2(G)$ of a graph $G$ are called \textit{component-homotopic} \cite{F} if one can be transformed to the other by a sequence of ambient isotopies and self-crossing changes. Here the self-crossing change of a spatial graph is a crossing change between arcs belonging to edges of the same component. 
\par 
The CL-homotopy theory of colored links was studied in \cite{FT} by using the \textit{colored Milnor groups} (CMGs). Fleming \cite{F} defined the notion of the component-homotopy for spatial graphs and studied it through the CMGs and CL-homotopy classes of constituent links of spatial graphs. He defined numerical invariant for the component-homotopy classes. In \cite{NM}, the component-homotopy classes of 3-component spatial graphs are classified by giving their structures. 

\begin{remark}
Taniyama \cite{T} also generalized the notion of link-homotopy for spatial graphs by two ways. Two spatial embeddings of a graph are \textit{edge-homotopic} if they are transformed to each other by a sequence of crossing changes between arcs belonging to the same edges and ambient isotopies. Two spatial embeddings of a graph are \textit{vertex-homotopic} is they are transformed to each other by a sequence of crossing changes between arcs belonging to adjacent edges of vertices and ambient isotopies. Note that, in \cite{T}, the edge-homotopy and the vertex-homotopy were called homotopy and weak-homotopy of graphs respectively. The edge-homotopy classes of the complete graph $K_4$ with four vertices was classified in \cite{Ni}. 
\end{remark}

\par
In this paper, we naturally generalize the Habegger-Lin theory \cite{HL} to colored string links. 
Let $\mathcal{CH}({\bm l})$ (resp. $\mathcal{CL}({\bm l})$) be the set of the CL-homotopy classes of colored string links (resp. colored links) with a component decomposition ${\bm l}=(1_{l_1},\dots,m_{l_m})$. 
A \textit{bouquet graph} is a graph with exactly one vertex. We call a graph $\Gamma$ a \textit{B-graph} if $\Gamma$'s components are all bouquet graphs and a spatial embedding of a B-graph a \textit{spatial B-graph}. 
Let $\mathcal{G}({\bm l})$ be the set of component-homotopy classes of $m$-component spatial B-graphs whose $i$-th component has $l_i$ edges.  
 We show that $\mathcal{CH}({\bm l})$ can be decomposed to reduced colored free groups (see Section \ref{classifications}). 
We define the Milnor invariants for colored string links and classify $\mathcal{CH}({\bm l})$. 
Then we classify $\mathcal{CL}({\bm l})$ by giving a Markov-type theorem for $\mathcal{CH}({\bm l})$. Moreover, by taking a special closure (we call it a G-closure) of colored string links, we also give a Markov-type theorem for  $\mathcal{G}({\bm l})$. This induces a classification of the component-homotopy classes of spatial graphs. \\

\par
This paper is organised as follows. In Section 2, we review the Habegger-Lin theory which classifies $\mathcal{L}(n)$ and $\mathcal{H}(n)$. In Section 3, we classify $\mathcal{CH}({\bm l})$ by applying the Habegger-Lin theory. Then, by considering closures of colored string links, we classify $\mathcal{CL}({\bm l})$ and $\mathcal{G}({\bm l})$. That induces a classification of the component-homotopy classes of spatial graphs. 

\section*{acknowledgement}
Y.K. is supported by ACT-X and JSPS KAKENHI, Grant-in-Aid for Early-Career Scientists Grant Number 20K14322, and by RIKEN iTHEMS Program and World Premier International Research Center Initiative (WPI) program, International Institute for Sustainability with Knotted Chiral Meta Matter (WPI-SKCM$^2$).

\section{Habegger and Lin's theory} \label{ColStringLink}

In this section, we review the Habegger-Lin theory which classifies $\mathcal{H}(n)$ and $\mathcal{L}(n)$. 

\subsection{string links}
Let $p_1, \cdots , p_n$ be points lying in order on the $x$-axis on the interior of the unit disk $D^2$. 
An {\it $n$-string link} $\sigma=\sigma_1 \cup \cdots \cup \sigma_n$ is a proper embedding of $n$ disjoint unit intervals $I_1, \cdots , I_n$ to $D^2 \times [0,1]$ such that for each $i=1, \cdots, n$, $\sigma_i(0)=(p_i, 0)$ and $\sigma_i(1)=(p_i, 1)$, where $\sigma_i$ is called the $i$-th string of the string link $\sigma$ (see Figure \ref{stringlink} left). 
Each string of a string link inherits an orientation from the usual orientation of the interval.
Composition of $n$-string links is defined as follows.
Let $\sigma=\sigma_1 \cup \cdots \cup \sigma_n$ and $\sigma'=\sigma'_1 \cup  \cdots \cup \sigma'_n$ be two string links. 
Then the {\it composition}  $\sigma\sigma'= (\sigma\sigma')_1 \cup  \cdots \cup (\sigma\sigma')_n$ of  $\sigma$ and  $\sigma'$ is the string link defined by $(\sigma\sigma')_i =h_1(\sigma_i) \cup h_2(\sigma'_i)$ for each $i=1, \cdots, n$, where $h_1,h_2: D^2 \times [0,1] \rightarrow D^2 \times [0,1]$ are embeddings defined by 
\[ h_1(p,t)=(p,\frac{1}{2}t) \text{ and } h_2(p,t)=(p,\frac{1}{2}+\frac{1}{2}t)  \]
for any $p \in D^2$ and $t \in [0,1]$ (see Figure \ref{stringlink} right). 
The {\it trivial} $n$-string link ${\bm 1}$ consists of $\sigma (I_i)={p_i} \times [0,1]$ for each $i=1, \cdots, n$.  
It is known that for fixed $n$, the set of link-homotopy classes $\mathcal{H}(n)$ of $n$-string links forms a group with multiplication induced by the composition of string links and the link-homotopy class of ${\bm 1}$ as the identity element.  

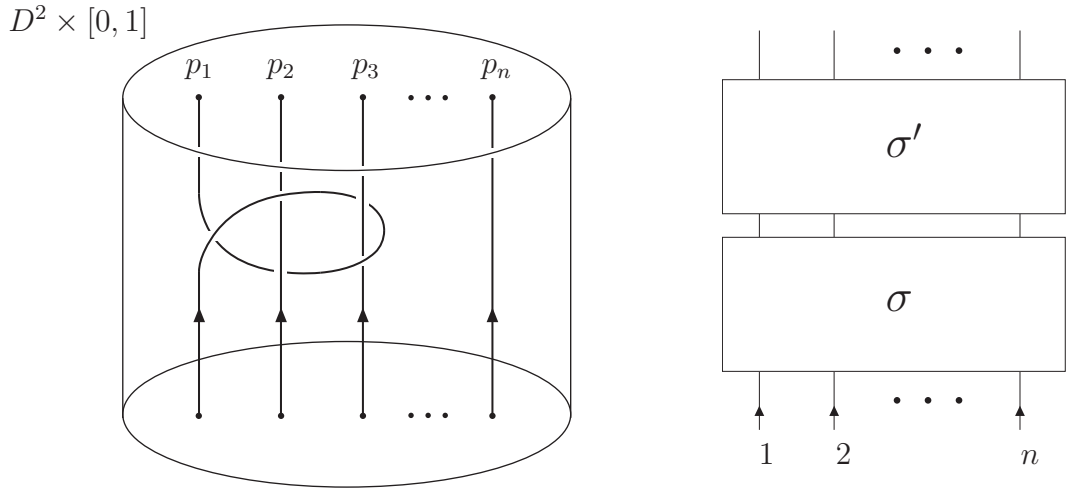
\begin{figure}[h]
$$
\raisebox{-22 pt}{\begin{overpic}[width=170pt]{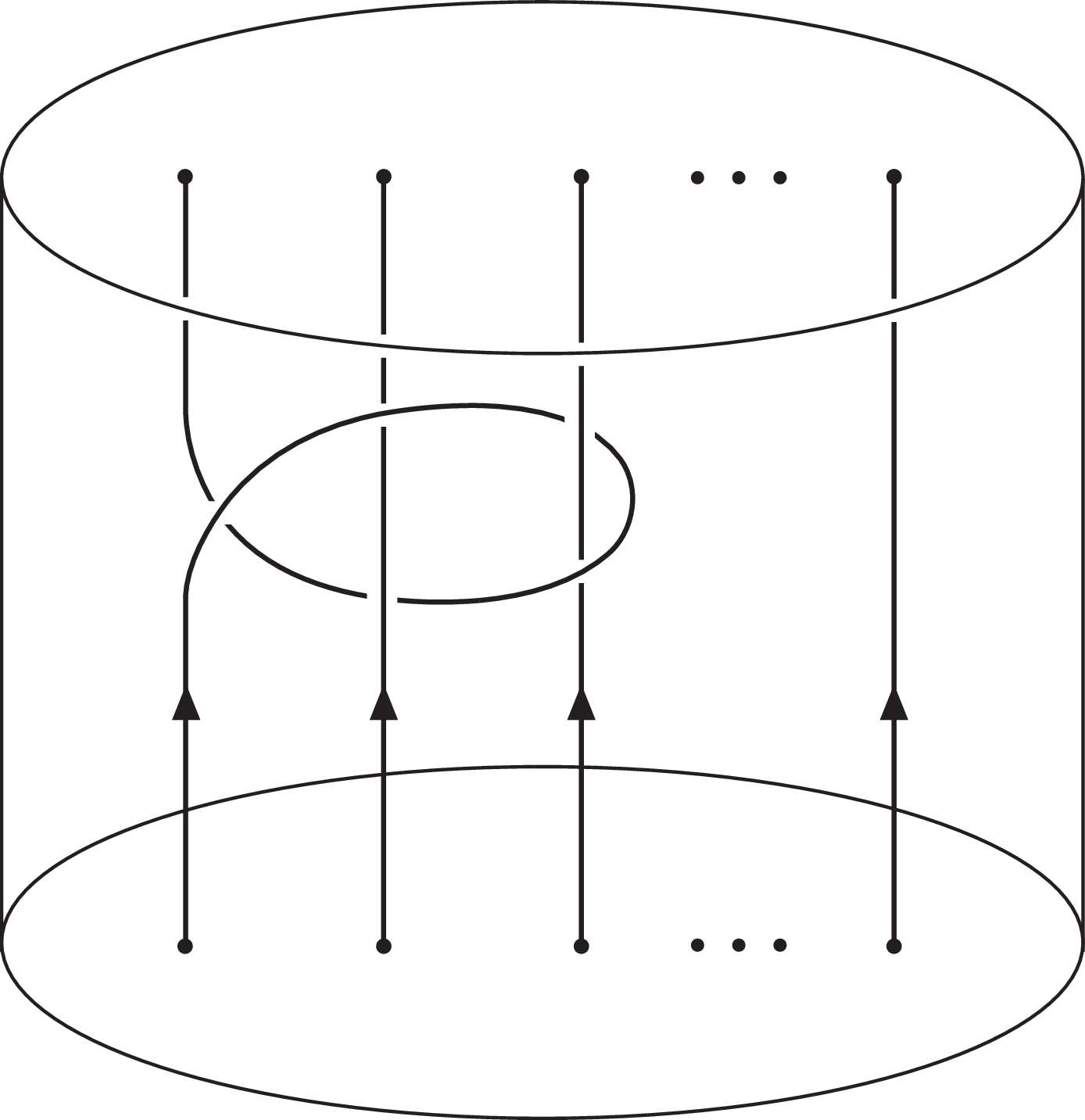}
\put(24,157){$p_1$} \put(55,157){$p_2$} 
\put(87,157){$p_3$} \put(136,157){$p_n$} 
\put(-43,173){$D^2\times [0,1]$} 
\end{overpic}} 
\hspace{2.0cm}
\raisebox{-0 pt}{\begin{overpic}[width=130pt]{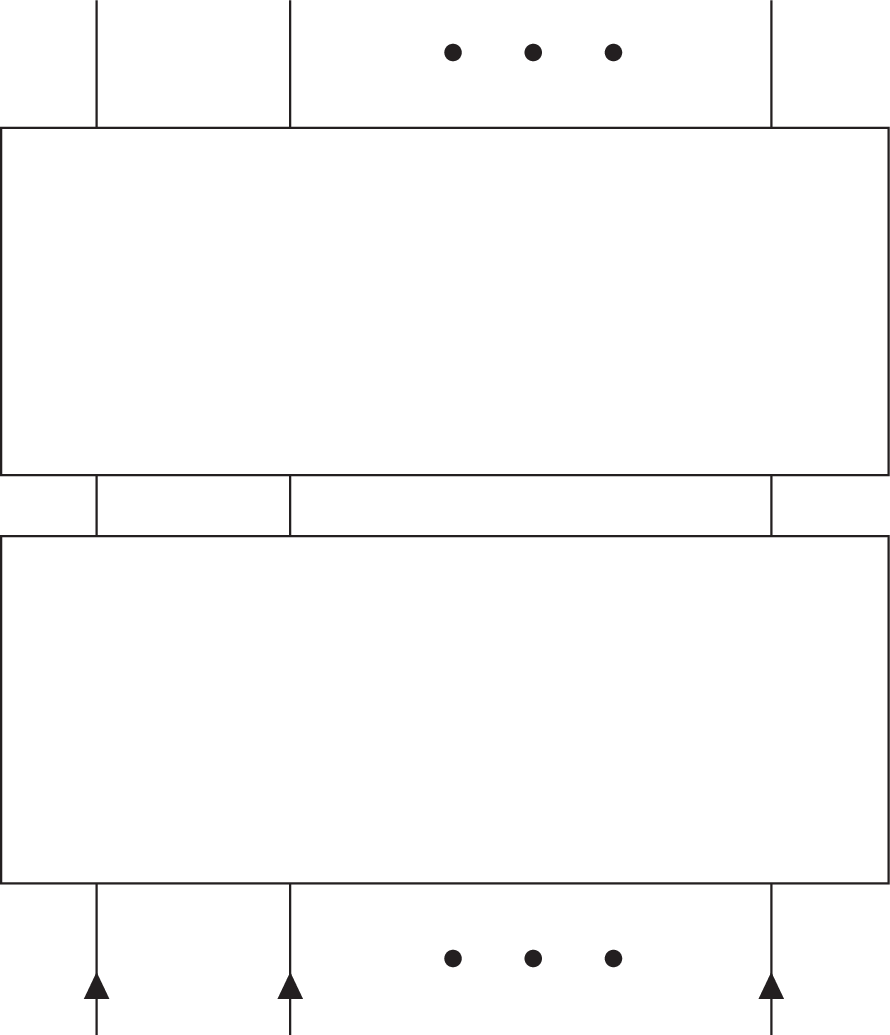}
\put(62,45){\Large $\sigma$} \put(62,103){\Large $\sigma'$}
\put(14,-13){$1$} \put(43,-13){$2$} 
\put(113,-13){$n$} 
\end{overpic}} 
$$ 
\vspace{0.0cm}
\caption{An $n$-string link and a composition.}\label{stringlink}
\end{figure}

\subsection{The link-homotopy classes of string links}
Let $F(n)$ be the free group with $n$ generators $x_i$ $(1\leq i\leq n)$. We identify the generators of $F(n)$ with those of the fundamental group of
$D^2\setminus \{p_1, \dots, p_{n-1}\}$ in Figure \ref{RF-generator}. 
Let $RF(n)$ be a \textit{reduced free group} which is $F(n)$ modulo the relations $[x_i,x^g_i]=1$ for each $i$ and any $g\in F(n)$, where $[a,b]=aba^{-1}b^{-1}$ is a commutator and $x^g_i=gx_ig^{-1}$ is a conjugation of $x_i$ by $g$. Then Habegger and Lin showed a group structure of $\mathcal{H}(n)$ as follows. 

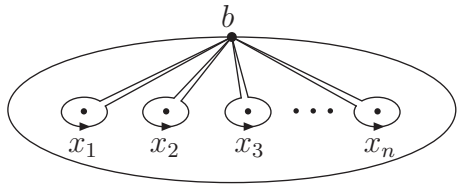
\begin{figure}[h]
$$
\raisebox{-22 pt}{\begin{overpic}[width=170pt]{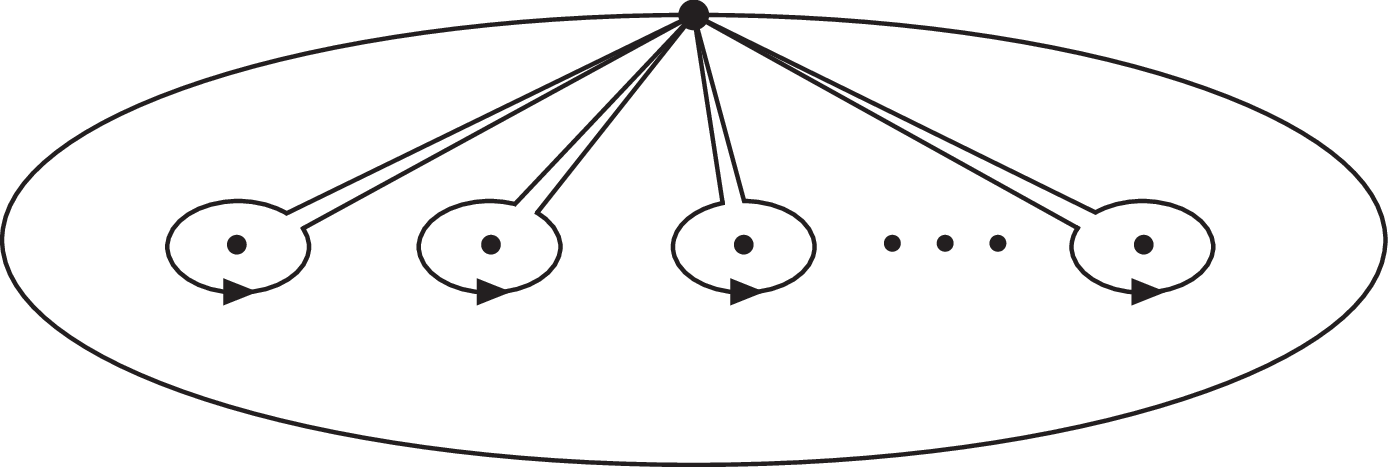}
\put(81,59){$b$} 
\put(23,12){$x_1$} \put(54,12){$x_2$} 
\put(86,12){$x_3$} \put(135,12){$x_n$}
\end{overpic}} 
$$ 
\vspace{0.0cm}
\caption{Generators of $\pi_1(D^2\setminus \{p_1, \dots, p_n\})$.}\label{RF-generator}
\end{figure}

\begin{theorem}[{\cite{HL}}] \label{SLclassify}
There is a split short exact sequence, 
$$1\to RF(n-1)\overset{}{\to} \mathcal{H}(n)\to \mathcal{H}(n-1)\to 1.$$
Then,
\begin{equation*}
\mathcal{H}(n)=\mathcal{H}(n-1)\ltimes RF(n-1).  
\end{equation*}
\end{theorem}
In Theorem \ref{SLclassify}, the group homomorphism $RF(n-1)\overset{}{\to} \mathcal{H}(n)$ maps $b\in RF(n-1)$ to the string link with the trivial first $n-1$ components and the $n$-th component 
determined from $b$ by treating it an element in $D^2\setminus \{p_1, \dots, p_n\}$, $\mathcal{H}(n)\to \mathcal{H}(n-1)$ is defined by omitting the $n$-th component and the splitting homomorphism $\mathcal{H}(n-1)\to \mathcal{H}(n)$ adds a trivial $n$-th component to an $(n-1)$-component string link. Using Theorem \ref{SLclassify} recursively, we have a decomposition of $\mathcal{H}(n)$ to reduced groups $RF(i)$, that is
\begin{equation*}
\mathcal{H}(n)=(\cdots(RF(1)\ltimes RF(2))\ltimes \cdots)\ltimes RF(n-1). \label{decomposition02}
\end{equation*}

\begin{remark}
From Theorem \ref{SLclassify}, for any $i$ $(1\leq i\leq n)$, we have a decomposition
\begin{equation}
\mathcal{H}(n)=\mathcal{H}(1,\dots,\hat{i},\dots,n)\ltimes RF(1,\dots,\hat{i},\dots,n), \label{decomposition01} 
\end{equation}
where $\mathcal{H}(1,\dots,\hat{i},\dots,n)$ is the set of link-homotopy classes of the $(n-1)$-component sub-string links obtained by omitting the $i$-th components from $n$-component string links and $RF(1,\dots,\hat{i},\dots,n)$ is the reduced free group generated by $x_j$ $(1\leq j\leq n, j\neq i)$. Moreover,
for any permutation $i_1i_2\dots i_n$ of $12\dots n$, we have a decomposition 
\begin{equation*}
\mathcal{H}(n)=(\cdots(RF_{i_1}\ltimes RF_{i_1i_2})\ltimes \cdots)\ltimes RF_{i_1\dots i_{n-1}}, 
\end{equation*}
where $RF_{i_1\dots i_{j}}$ is the reduced free group generated by $x_{i_1}, \dots ,x_{i_j}$. 
\end{remark}

\par
We define a group action $\Sigma\cdot \sigma$ of $\Sigma\in\mathcal{H}(2n)$ to $\sigma\in\mathcal{H}(n)$. 
For convenience, we call the first $n$ components of string links in $\mathcal{H}(2n)$ 
the $\overline{n}, \overline{n-1}, \dots, \overline{1}$-th components in order and the last $n$ components $1, 2,\dots, n$-th components, see Figure \ref{H2-action} left. 
The action $\Sigma\cdot \sigma$ is defined as in Figure \ref{H2-action} right. 
Under the action, the $\overline{i}$-th component of $\Sigma$ becomes $i$-th component of the resulting string link. We choose the generators of the fundamental group of $D^2\setminus \{p_{\overline{n}},\dots,p_{\overline{1}},p_1,\dots,p_n\}$ as in Figure \ref{H2-generators} so that they are compatible with those of $D^2\setminus \{p_1, \dots, p_n\}$. 
\begin{figure}[h]
$$
\raisebox{10 pt}{\begin{overpic}[width=180pt]{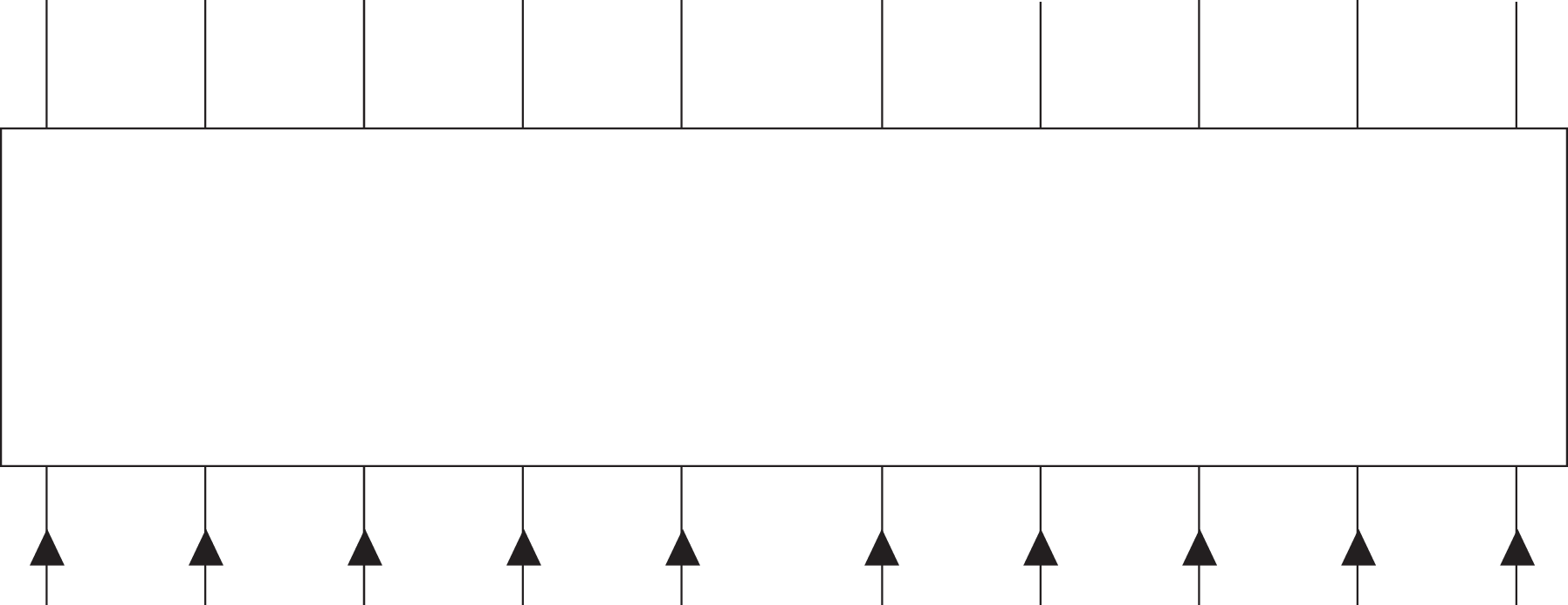}
\put(84,30){\large$\Sigma$}
\put(75,-14){$\overline{1}$} \put(57,-14){$\overline{2}$} 
\put(26,-11){$\cdots$} \put(2,-13){$\overline{n}$}
\put(98,-12){$1$} \put(116,-12){$2$} 
\put(140,-11){$\cdots$} \put(170,-11){$n$}
\end{overpic}} 
\hspace{2.0cm}
\raisebox{-20 pt}{\begin{overpic}[width=180pt]{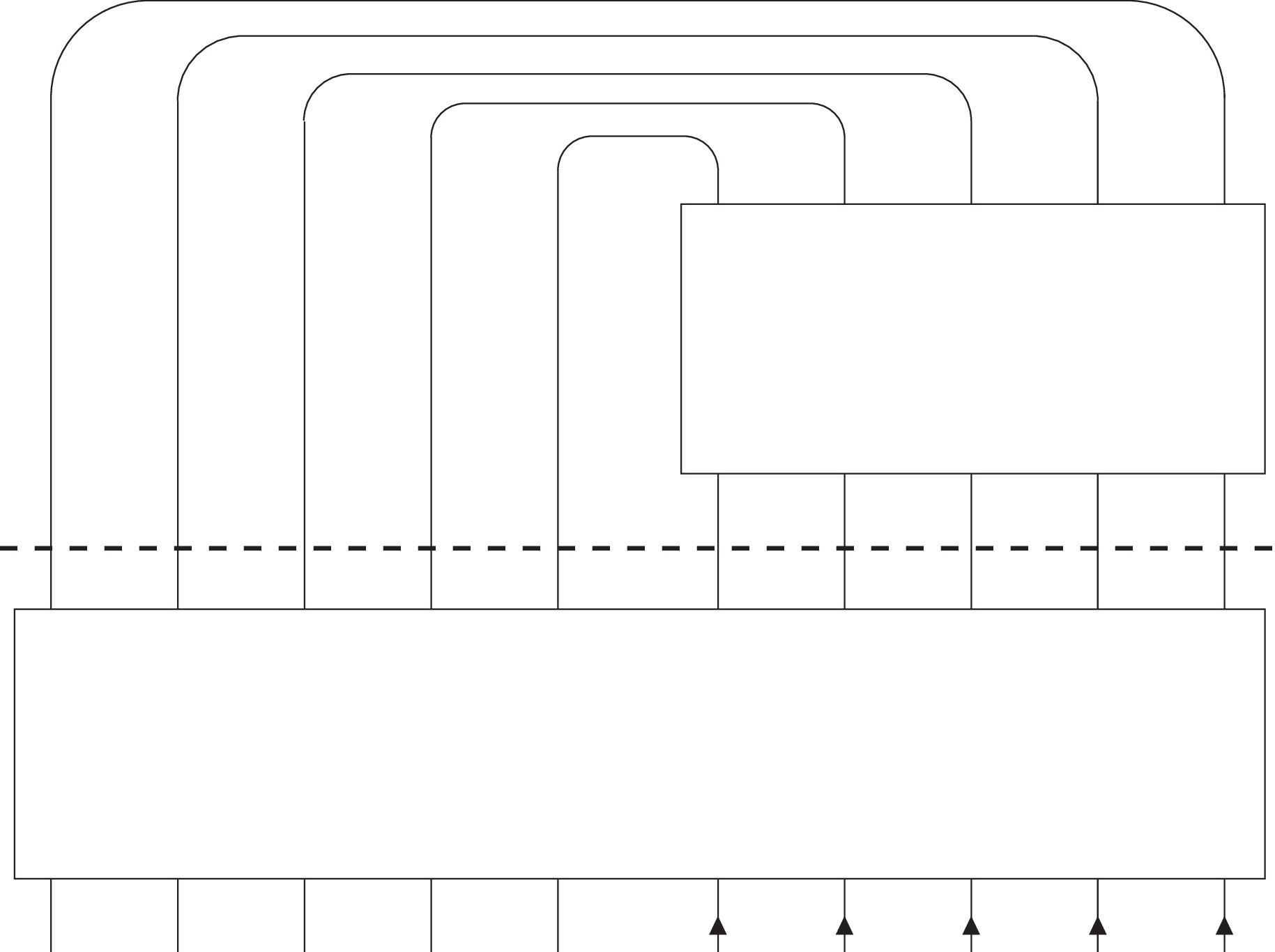}
\put(133,83){\large$\sigma$} \put(84,24){\large$\Sigma$}
\put(98,-12){$1$} \put(116,-12){$2$}
\put(139,-11){$\cdots$} \put(169,-11){$n$}
\end{overpic}} 
$$ 
\vspace{0.1cm}
\caption{An action of $\mathcal{H}(2n)$ on $\mathcal{H}(n)$.}\label{H2-action}
\end{figure}

\begin{figure}[h]
$$
\raisebox{10 pt}{\begin{overpic}[width=240pt]{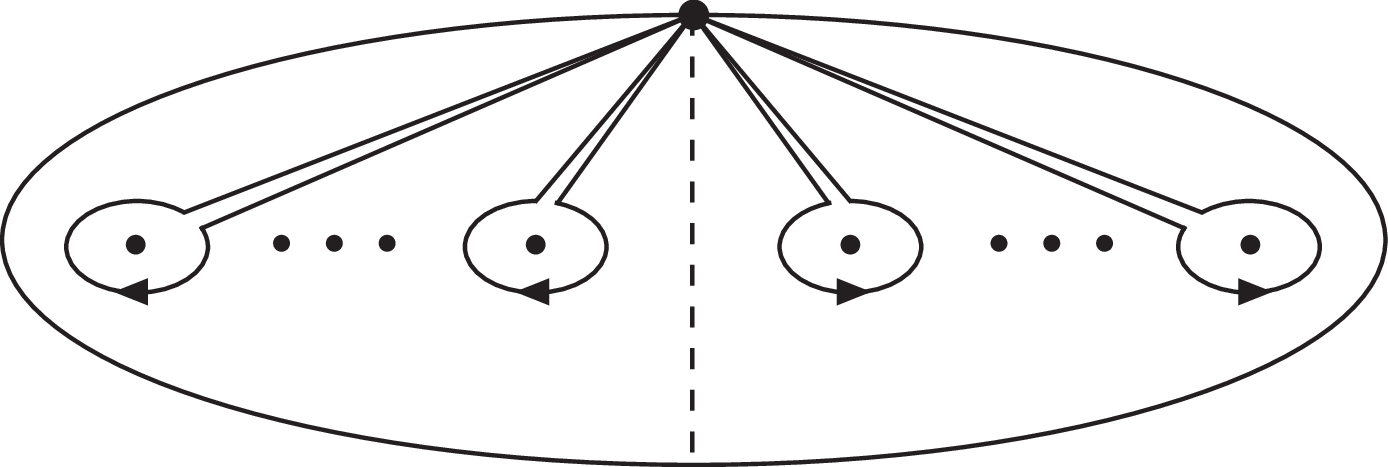}
\put(117,83){$b$} 
\put(90,15){$\overline{1}$} 
\put(21,18){$\overline{n}$}
\put(145,17){$1$} 
\put(213,20){$n$}
\end{overpic}} 
$$ 
\vspace{-1.0cm}
\caption{Generators of $\pi_1(D^2\setminus \{p_{\overline{n}},\dots,p_{\overline{1}},p_1,\dots,p_n\})$.}\label{H2-generators}
\end{figure}
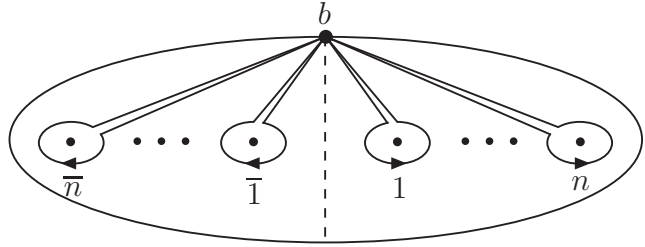

\par
Habegger and Lin also proved a Markov-type theorem for the set of the link-homotopy classes of links and classified it. 
Let $\mathcal{S}(n)$ be a subgroup of $\mathcal{H}(2n)$ whose elements $s$ satisfy $s\cdot {\bm 1} = {\bm 1}$ for ${\bm 1}\in \mathcal{H}(n)$. 

\begin{theorem}[\cite{HL}] \label{Lclassify}
Let $\sigma$ and $\sigma'$ be in $\mathcal{H}(n)$. Then the closures of $\sigma$ and $\sigma'$ are link-homotopic if and only if there is an element $s\in\mathcal{S}(n)$ such that $s\cdot\sigma=\sigma'$. 
\end{theorem}

Habegger and Lin gave generators of $\mathcal{S}(n)$, which are 
$(\overline{x}_i, \overline{x}_i)_j$, $(x_i, x_i)_j$ and $(\overline{x}_i, x_i)_j$ in Figure \ref{S-generators}. The actions of $(\overline{x}_i, x_i)_j$ $(1\leq i, j\leq n, i\neq j)$ generate the \textit{conjugations} $\sigma^h=h\sigma h^{-1}$ ($h\in \mathcal{H}(n)$) of string links.  
From the decomposition ($\ref{decomposition01}$), $\sigma\in \mathcal{H}(n)$ has a form $ \sigma= \theta g$, where $\theta\in \mathcal{H}(1,\dots,\hat{i},\dots,n)$ and $g\in RF(1,\dots,\hat{i},\dots,n)$. The generators $(\overline{x}_j, \overline{x}_j)_i$ and $(x_j, x_j)_i$ acts $\sigma= \theta g$ as follows (\cite{HL}), 
$$(\overline{x}_j, \overline{x}_j)_i\cdot \theta g=\theta g^{x_j}, 
\quad (x_j, x_j)_i\cdot \theta g=\theta g^{\theta^{-1}x_j\theta}.$$
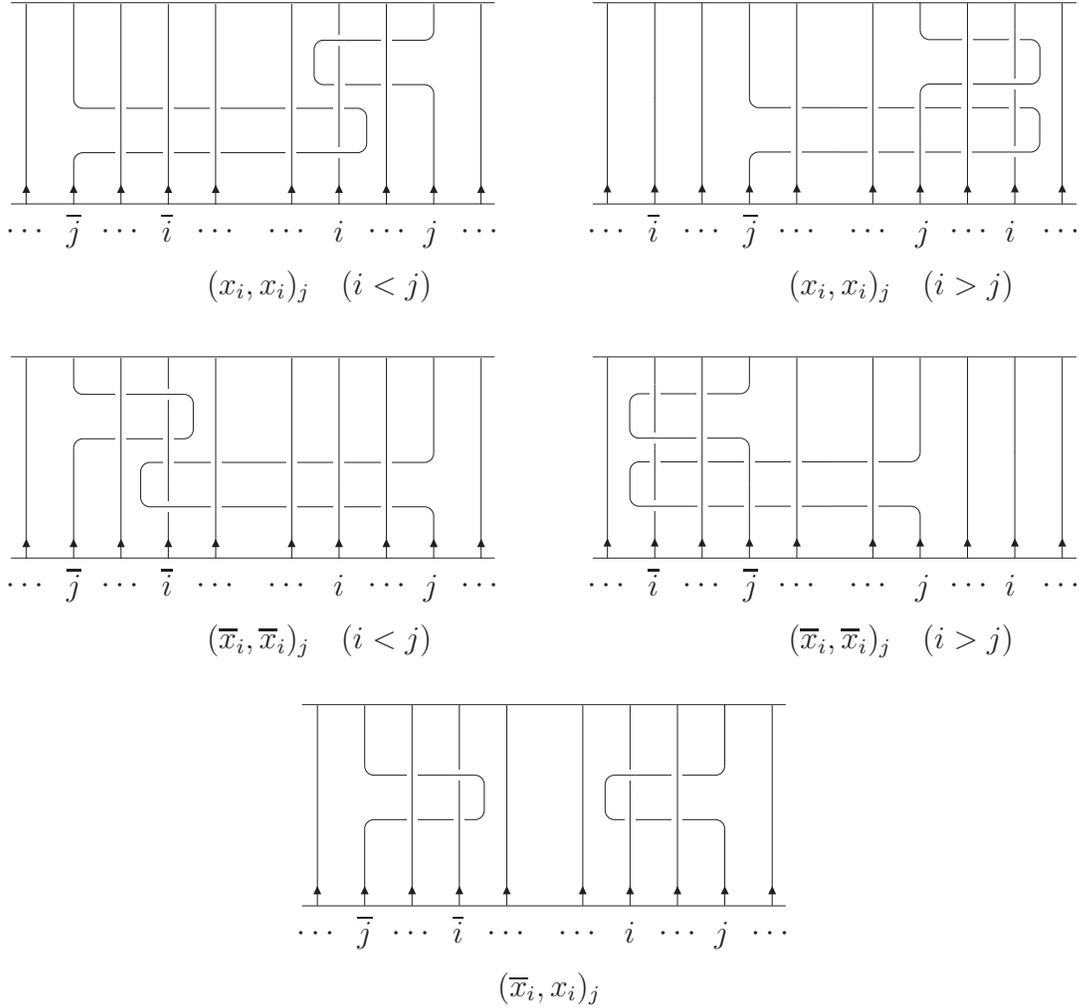
\begin{figure}[h]
\vspace{0.4cm}
$$
\raisebox{-20 pt}{\begin{overpic}[width=183pt]{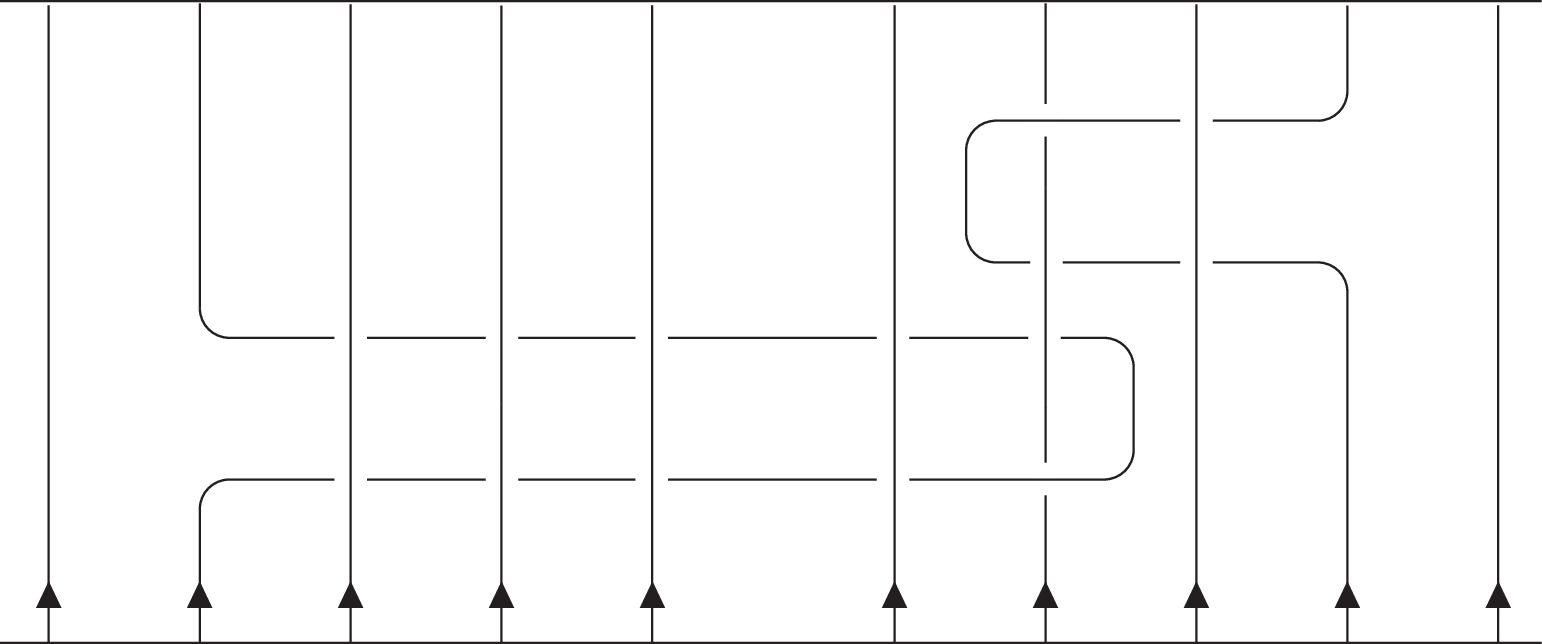}
\put(-2,-13){$\cdots$} \put(21,-14){$\overline{j}$} \put(34,-13){$\cdots$}\put(57,-14){$\overline{i}$} \put(70,-13){$\cdots$}
\put(97,-13){$\cdots$} \put(122,-14){$i$} \put(135,-13){$\cdots$} \put(156,-14){$j$} \put(170,-13){$\cdots$}
\put(74,-34){$(x_i,x_i)_j\quad (i<j)$}
\end{overpic}} 
\hspace{1.3cm}
\raisebox{-20 pt}{\begin{overpic}[width=183pt]{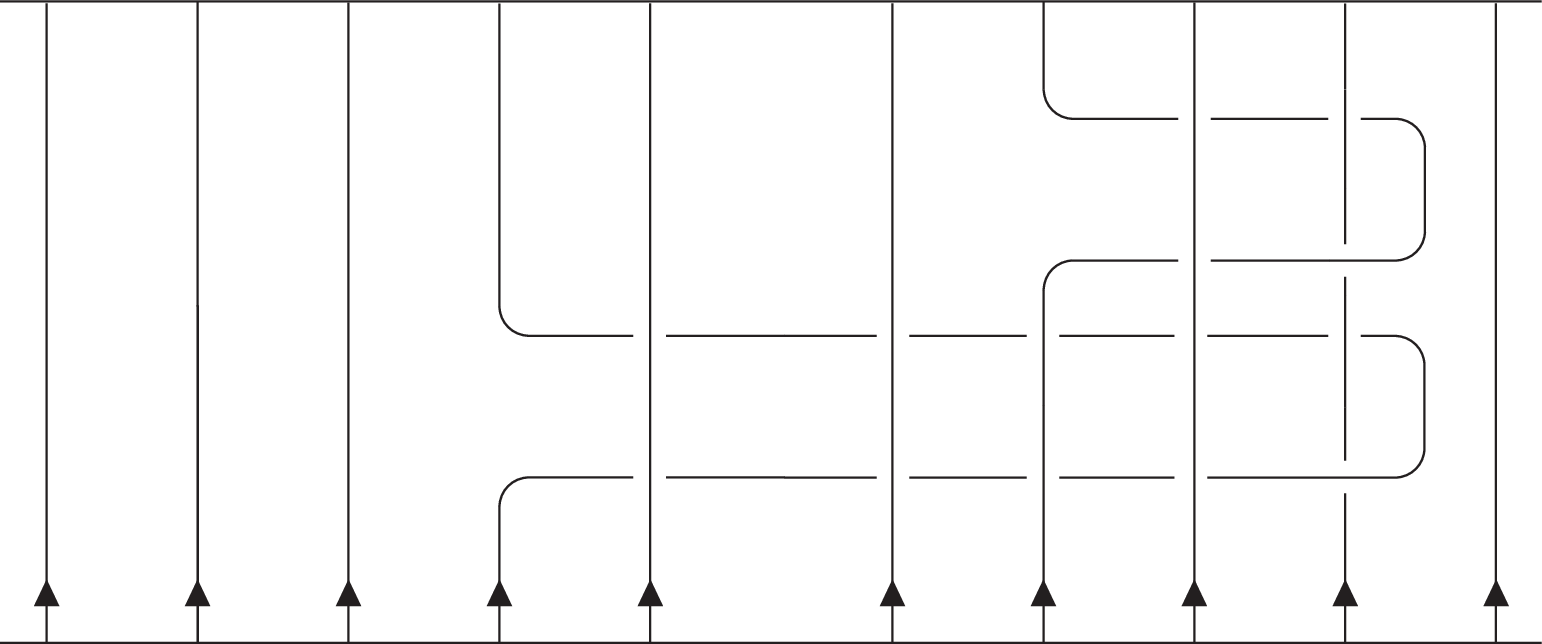}
\put(-2,-13){$\cdots$} \put(21,-14){$\overline{i}$} \put(34,-13){$\cdots$}\put(57,-14){$\overline{j}$} \put(70,-13){$\cdots$}
\put(97,-13){$\cdots$} \put(122,-14){$j$} \put(135,-13){$\cdots$} \put(156,-14){$i$} \put(170,-13){$\cdots$}
\put(74,-34){$(x_i,x_i)_j\quad (i>j)$}
\end{overpic}} 
$$
\vspace{1.8cm}
$$
\raisebox{-20 pt}{\begin{overpic}[width=183pt]{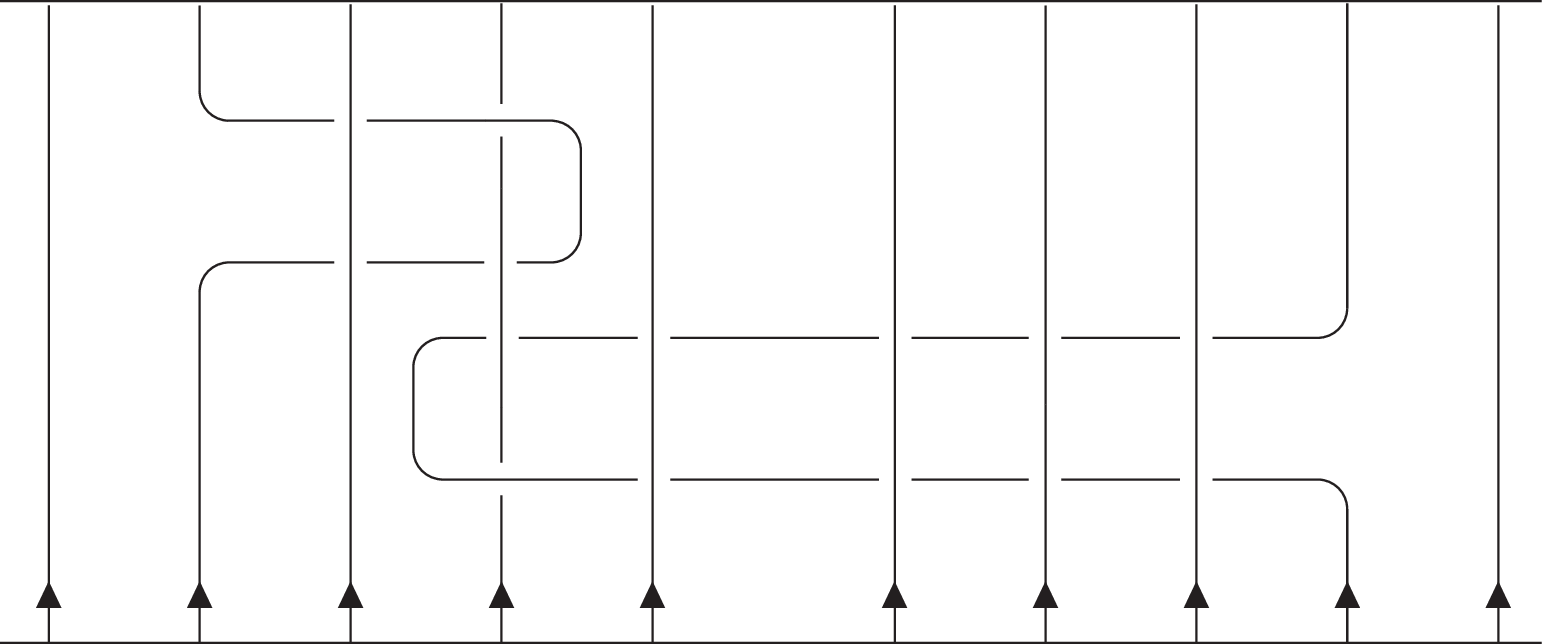}
\put(-2,-13){$\cdots$} \put(21,-14){$\overline{j}$} \put(34,-13){$\cdots$}\put(57,-14){$\overline{i}$} \put(70,-13){$\cdots$}
\put(97,-13){$\cdots$} \put(122,-14){$i$} \put(135,-13){$\cdots$} \put(156,-14){$j$} \put(170,-13){$\cdots$}
\put(74,-34){$(\overline{x}_i,\overline{x}_i)_j\quad (i<j)$}
\end{overpic}} 
\hspace{1.3cm}
\raisebox{-20 pt}{\begin{overpic}[width=183pt]{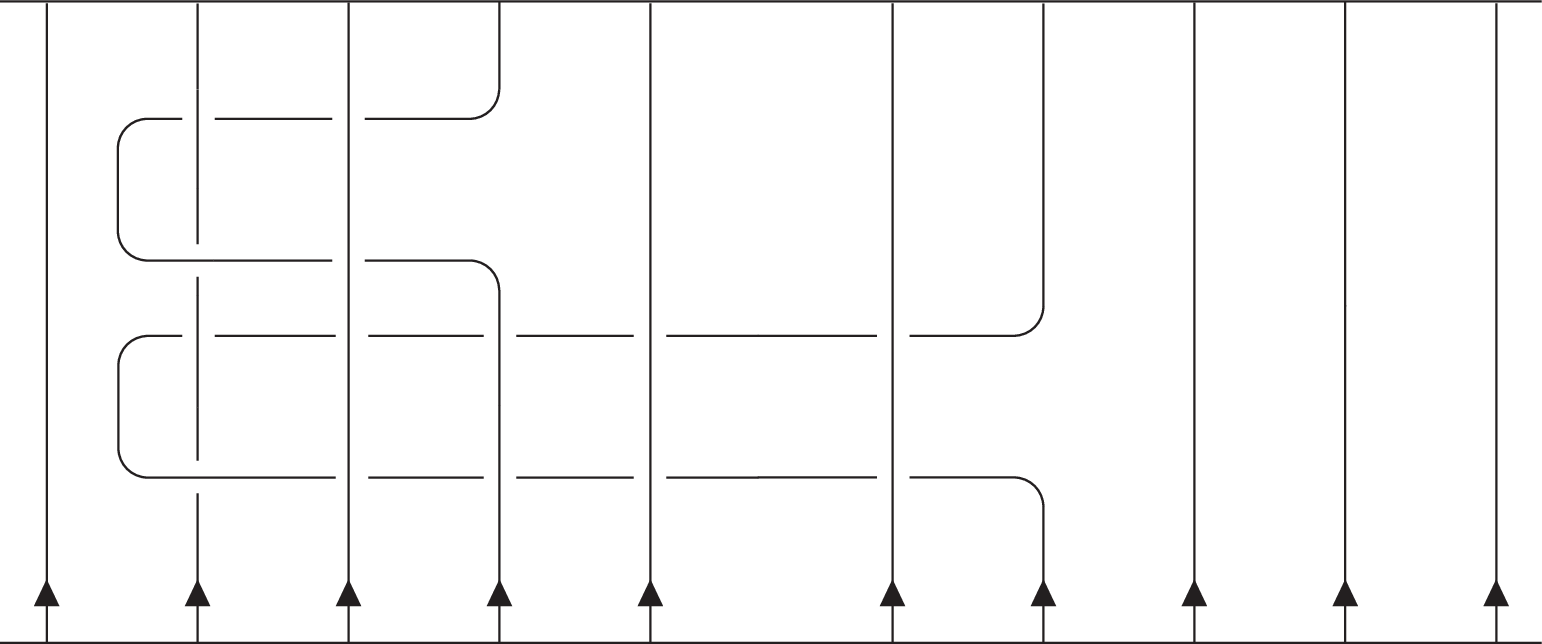}
\put(-2,-13){$\cdots$} \put(21,-14){$\overline{i}$} \put(34,-13){$\cdots$}\put(57,-14){$\overline{j}$} \put(70,-13){$\cdots$}
\put(97,-13){$\cdots$} \put(122,-14){$j$} \put(135,-13){$\cdots$} \put(156,-14){$i$} \put(170,-13){$\cdots$}
\put(74,-34){$(\overline{x}_i,\overline{x}_i)_j\quad (i>j)$}
\end{overpic}} 
$$ 
\vspace{1.8cm}
$$
\raisebox{-20 pt}{\begin{overpic}[width=183pt]{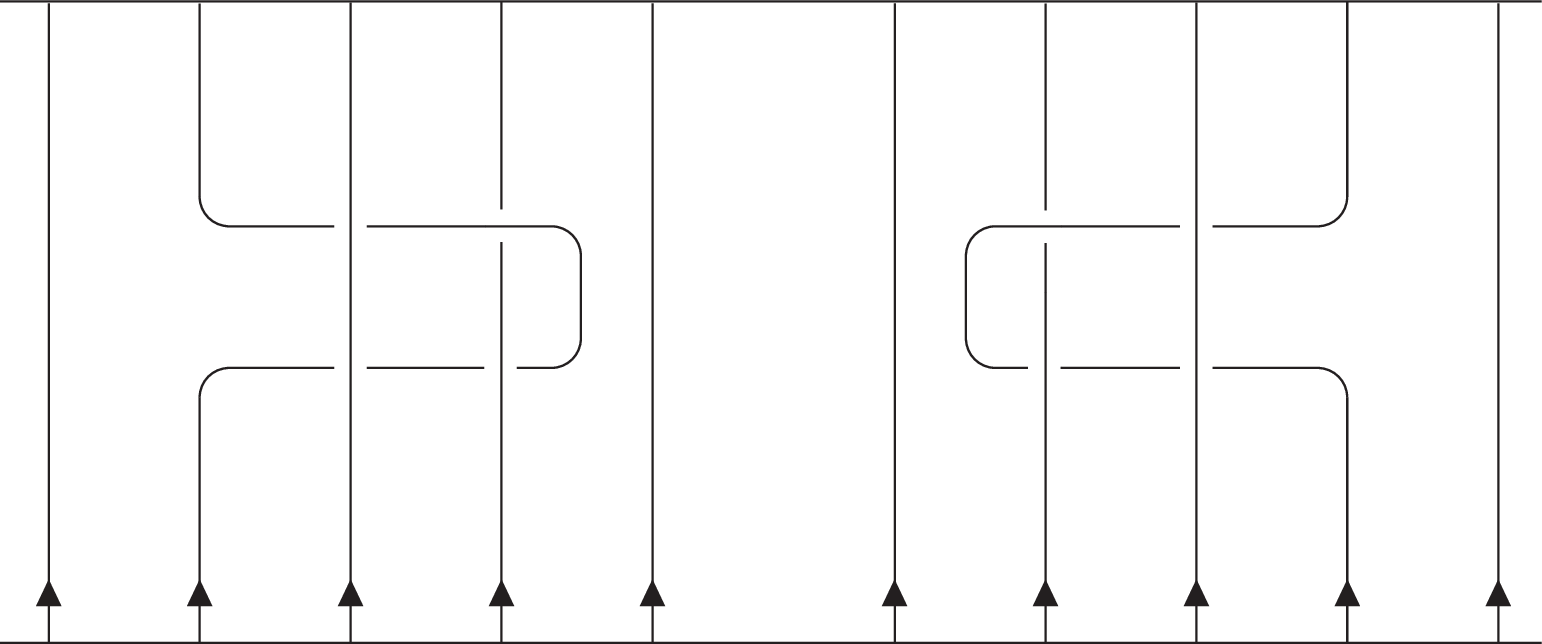}
\put(-2,-13){$\cdots$} \put(21,-14){$\overline{j}$} \put(34,-13){$\cdots$}\put(57,-14){$\overline{i}$} \put(70,-13){$\cdots$}
\put(97,-13){$\cdots$} \put(122,-14){$i$} \put(135,-13){$\cdots$} \put(156,-14){$j$} \put(170,-13){$\cdots$}
\put(74,-34){$(\overline{x}_i,x_i)_j$}
\end{overpic}}
$$ 
\vspace{1.0cm}
\caption{Generators of $\mathcal{S}(n)$.}\label{S-generators}
\end{figure}

For any $i$ $(1\leq i\leq n)$ and the above decomposition $\sigma=\theta g$, the elements $\theta g^{h}=\theta hgh^{-1}$ ($h\in RF(1,\dots,\hat{i},\dots,n)$) are called \textit{partial conjugations} of $\sigma$. Remark that $(\overline{x}_i, \overline{x}_i)_j$ $(1\leq i,j\leq n, i\neq j)$ generate the partial conjugations (including the actions of $(x_i, x_i)_j$). 
Hudges \cite{Hu} showed that $(\overline{x}_i, \overline{x}_i)_j$ also generate the conjugations. 

\begin{proposition}[\cite{HL}, \cite{Hu}] \label{L-kernel}
$\mathcal{S}(n)$ is generated by the elements $(\overline{x}_{i}, \overline{x}_{i})_{j}$.
\end{proposition}

\subsection{Milnor invariants} 
In \cite{HL}, Habegger and Lin also classified $\mathcal{H}(n)$ by the \textit{Milnor invariants} ${\mu}_S(I)$ for string links $S$. 
The Milnor invariants are defined as follows. A \textit{Magnus expansion} $\mathcal{F}$ is the group homomorphism 
$$\mathcal{F}:F(n)\to \langle\langle X_1,\dots,X_n\rangle\rangle_{\mathbb{Z}},$$
where $\mathcal{F}$ is defined by $x_i\mapsto 1+X_i$ and $\langle\langle X_1,\dots,X_n\rangle\rangle_{\mathbb{Z}}$ is the ring of formal power series in the non-commuting variables $X_i$ with integer coefficients. 
The $i$-th component of $S$ is represented by an element $\lambda_i$ in $RF(1,\dots,\hat{i},\dots,n)$ from the decomposition (\ref{decomposition01}). 
For a sequence $I=j_1\dots j_k i$ in $\{1,\dots,n\}$, the coefficient
$\mu_S(I)$ of the term $X_{j_1}\cdots X_{j_k}$ of $\mathcal{F}(\lambda_i)$ is well-defined and called the Milnor invariants of $S$ for $I$. The Milnor invariants ${\mu}_S(I)$ are isotopy invariants for $S$.
For a non-repeat sequence $I$, $\mu_S(I)$ is link-homotopy invariants and they classify $\mathcal{H}(n)$ \cite{HL} (see also \cite{Y, MY2}). 
\par
For the link case, we also define Milnor invariants $\overline{\mu}_L(I)$ (\cite{Mil2}) for a link $L$ in $S^3$. Let $\lambda_i$ be an element in $\pi_1(S^3\setminus L)$ which represents the preferred longitude
 of the $i$-th component of $L$ (i.e. the longitude whose linking number $0$ with the $i$-th component and whose orientation is the same as that of $i$-th component). We calculate $\mu_L(I)$ as above for the sequence  $I$ of component numbers. Let $\Delta(I)$ be a greatest common divisor of $\mu_L(J)$ for all sequence $J$ which obtained from $I$ by deleting more than or equal to one letter and permuting cyclically. Then ${\mu}_L(I)$ modulo $\Delta(I)$ is isotopy invariant for $L$ and denoted by $\overline{\mu}_L(I)$. For a non-repeat sequence $I$, $\overline{\mu}_L(I)$ is link-homotopy invariants for $L$. 
 \par
 We call the Milnor invariants $\mu_S(I)$ (resp. $\overline{\mu}_L(I)$) with non-repeat sequence $I$ \textit{Milnor homotopy invariants} for string link $S$ (resp. link $L$). The Milnor homotopy invariants classify $\mathcal{L}(n)$ for $n\leq 3$. For $n\geq 4$, it is known that there is a pair of links which are not link-homotopic and have the same link-homotopy invariants (see for example \cite{Le2, , KM2}).

\section{Classifications} \label{classifications}

\subsection{A classification of the link-homotopy classes of colored string links}
We take an analogue of the Habegger-Lin theory for colored string links. An example of a colored string link is in Figure \ref{H(S)-generator}, where the  $(i,j)$-th component connects points $p_{ij}$ in $D^2\times \{0\}$ and $D^2\times \{1\}$. 

\par
Let $\mathcal{CH}({\bm l})$ be the set of the CL-homotopy classes of colored string links with $m$ colors and a component decomposition ${\bm l}=(1_{l_1},\dots,m_{l_m})$, which has obviously the group structure.  

Let $F({\bm l})=F(1_{l_1},\dots,m_{l_m})$ be a free group generated by $x_{ij}$ ($1\leq i\leq k, 1\leq j \leq l_i$). 
Let $RCF({\bm l})$ be $F({\bm l})$ modulo relations such that conjugations of $x_{ij}$ and $x_{st}$ commute if $i=s$, i.e.  $x_{ij}^gx_{it}^h=x_{it}^hx_{ij}^g$ for $g, h\in F({\bm l})$, where $x_{ij}^g=gx_{ij}g^{-1}$ and $x_{it}^h=hx_{it}h^{-1}$. 
We call $RCF({\bm l})$ a \textit{reduced colored free group} of $F({\bm l})$. As an analogue of Theorem \ref{SLclassify}, we have a decomposition theorem for $\mathcal{CH}({\bm l})$.  


\begin{figure}[h]
$$
\raisebox{-22 pt}{\begin{overpic}[width=220pt]{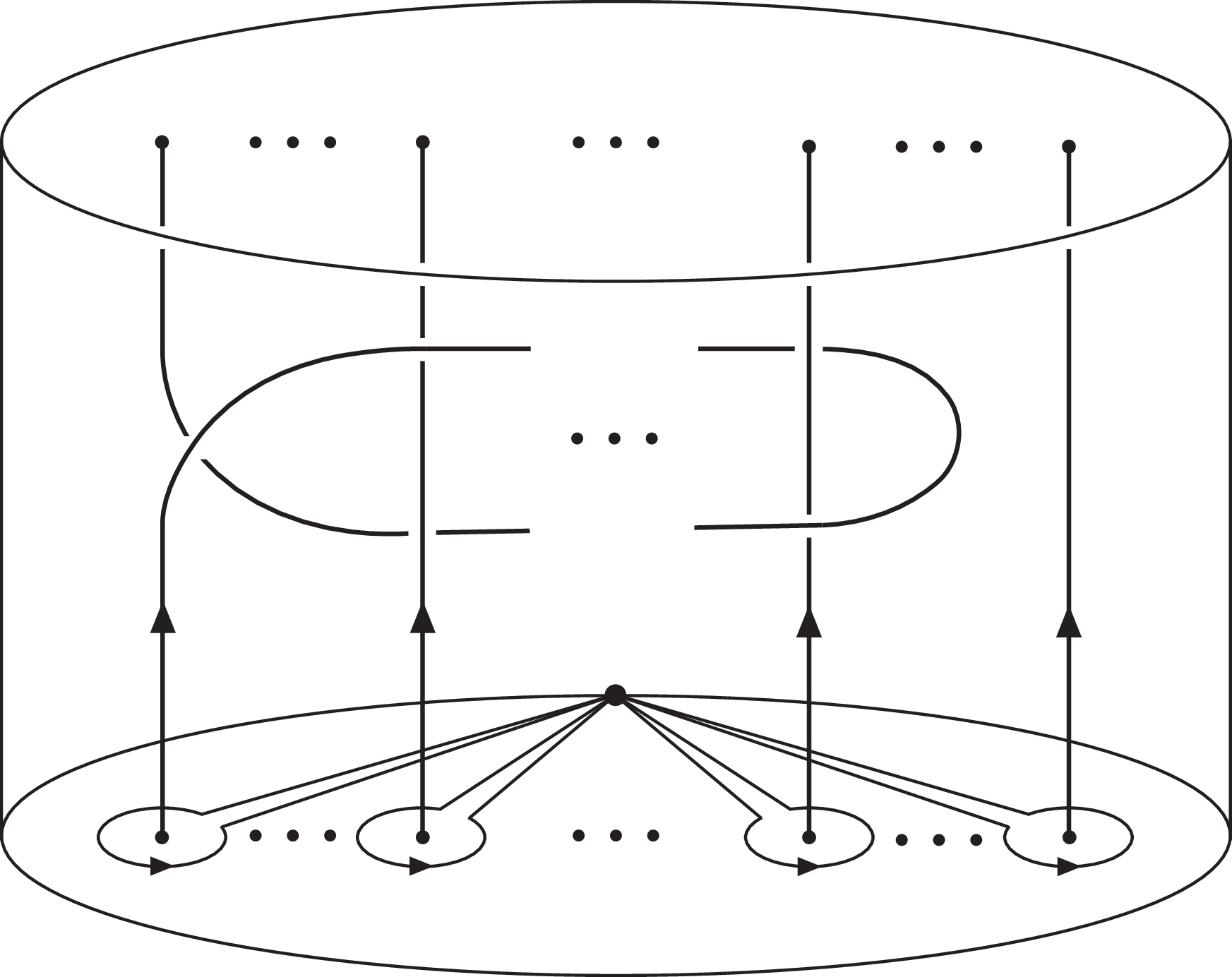}
\put(107,55){$b$} 
\put(23,11){$x_{11}$} \put(66,11){$x_{1l_1}$} 
\put(135,11){$x_{m1}$} \put(173,13){$x_{ml_m}$}
\put(24,157){$p_{11}$} \put(69,157){$p_{1l_1}$} 
\put(137,157){$p_{m1}$} \put(177,157){$p_{ml_m}$} 
\put(-43,173){$D^2\times [0,1]$} 
\end{overpic}} 
$$ 
\vspace{0.0cm}
\caption{A colored string link and generators of $\pi_1(D^2\setminus\{p_{11},\dots,p_{ml_m}\})$.}\label{H(S)-generator}
\end{figure}
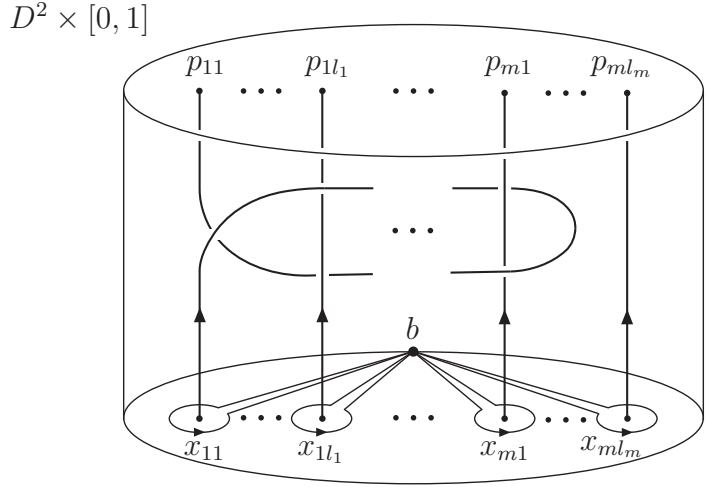

\begin{theorem} \label{CSLclassify} 
For the group $\mathcal{CH}({\bm l})$ of CL-homotopy classes of colored string links with $m$ colors and a component decomposition ${\bm l}=(1_{l_1},\dots,m_{l_m})$, there is a split short exact sequence 
$$1\to RCF({\bm l}_{m-1})^{l_m}\overset{}{\to} \mathcal{CH}({\bm l})\to \mathcal{CH}({\bm l}_{m-1})\to 1,$$
where ${\bm l}_i=(1_{l_1},\dots,i_{l_i})$. Then, 
$\mathcal{CH}({\bm l})$ is decomposed to a semidirect product as follows.
\begin{equation}
\mathcal{CH}({\bm l})=\mathcal{CH}({\bm l}_{m-1})\ltimes RCF({\bm l}_{m-1})^{l_m}. \label{CLdecomposition}
\end{equation}
\end{theorem}

In Theorem \ref{CSLclassify}, the homomorphism $\mathcal{CH}({\bm l})\to \mathcal{CH}({\bm l}_{m-1})$ is defined by omitting the sub-string link $\sigma_m$, the split homomorphism $\mathcal{CH}({\bm l}_{m-1})\to \mathcal{CH}({\bm l})$ adds the trivial $l_m$ strings as the $m$-th sub-string link to a colored string link with $m-1$ colors and $RCF({\bm l}_{m-1})^{l_m}\overset{}{\to} \mathcal{CH}({\bm l})$ is defined by adding $l_m$ strings as the $m$-th sub-string link which are determined by an element of $RCF({\bm l}_{m-1})^{l_m}$ to the trivial colored string link with the component decomposition ${\bm l}_{m-1}$.

\begin{remark}
By using the above decomposition recursively, we have
$$\mathcal{CH}({\bm l})=(\cdots(\,RCF({\bm l}_{1})^{l_2}\ltimes RCF({\bm l}_{2})^{l_3}\,)\ltimes \cdots)\ltimes RCF({\bm l}_{m-1})^{l_m}.$$
For a permutation $i_1\cdots i_m$ of $1\cdots m$, similarly we have a decomposition: 
$$\mathcal{CH}({\bm l})=(\cdots(RCF(l_{i_1})^{l_{i_2}}\ltimes RCF(l_{i_1},l_{i_2})^{l_{i_3}})\cdots)\ltimes RCF(l_{i_1},\dots,l_{i_{m-1}})^{l_{i_m}}.$$
Here $RCF(l_{i_1},\dots,l_{i_k})$ is generated by $x_{st}$ ($s\in \{i_1,\dots,i_k\}, 1\leq t\leq l_{s}$). 
\end{remark}

\subsection{A classification of the link-homotopy classes of colored links}
For ${\bm l}=(1_{l_1},\dots,m_{l_m})$, put $\overline{\bm l}{\bm l}=(\overline{m}_{l_{\overline{m}}},\dots,\overline{1}_{l_{\overline{1}}},1_{l_1},\dots,m_{l_m})$, where $l_i=l_{\overline{i}}$. We sometimes abbreviate a component number $(i,j)$ to $ij$ and  $(\overline{i},j)$ to $\overline{ij}$. 
For $\Sigma\in$ $\mathcal{CH}(\overline{\bm l}{\bm l})$, we define a group  action of $\Sigma$ to $\sigma\in \mathcal{CH}({\bm l})$ as a composition $\Sigma\cdot\sigma$ as in Figure \ref{CLaction}. 
Here we forget the orientation of $\Sigma$ once and give the orientation compatible with that of $\sigma$ and identify $\overline{ij}$ to $ij$. 
We also choose generators of $\pi_1(D^2\setminus\{p_{\overline{ml_m}},\dots, p_{\overline{11}}, p_{11},\dots,p_{ml_m}\})$ as the loops in Figure \ref{CLH2-generators}. 

\begin{figure}[h]
$$
\raisebox{10 pt}{\begin{overpic}[width=180pt]{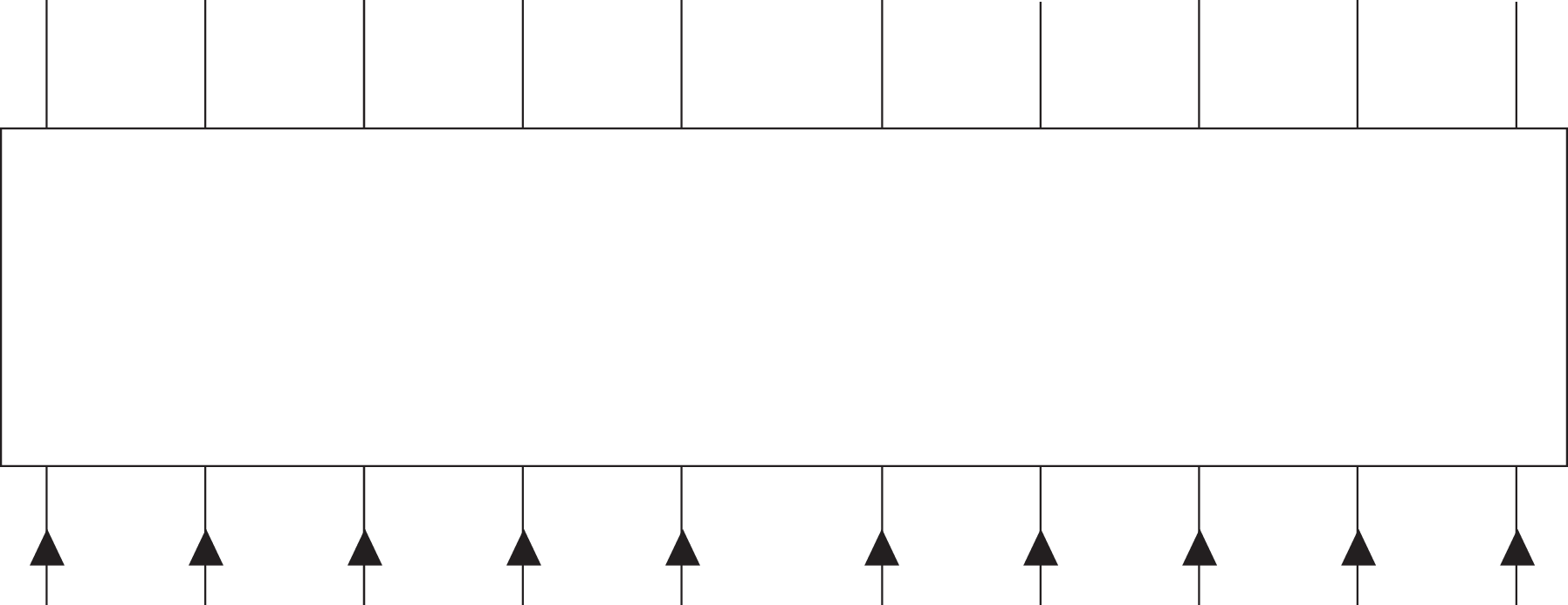}
\put(84,30){\large$\Sigma$}
\put(73,-13){$\overline{11}$} \put(53,-13){$\overline{12}$} 
\put(26,-11){$\cdots$} \put(-6,-14){$\overline{ml_m}$}
\put(95,-11){$11$} \put(114,-12){$12$} 
\put(139,-11){$\cdots$} \put(167,-11){$ml_m$}
\end{overpic}} 
\hspace{2.0cm}
\raisebox{-20 pt}{\begin{overpic}[width=180pt]{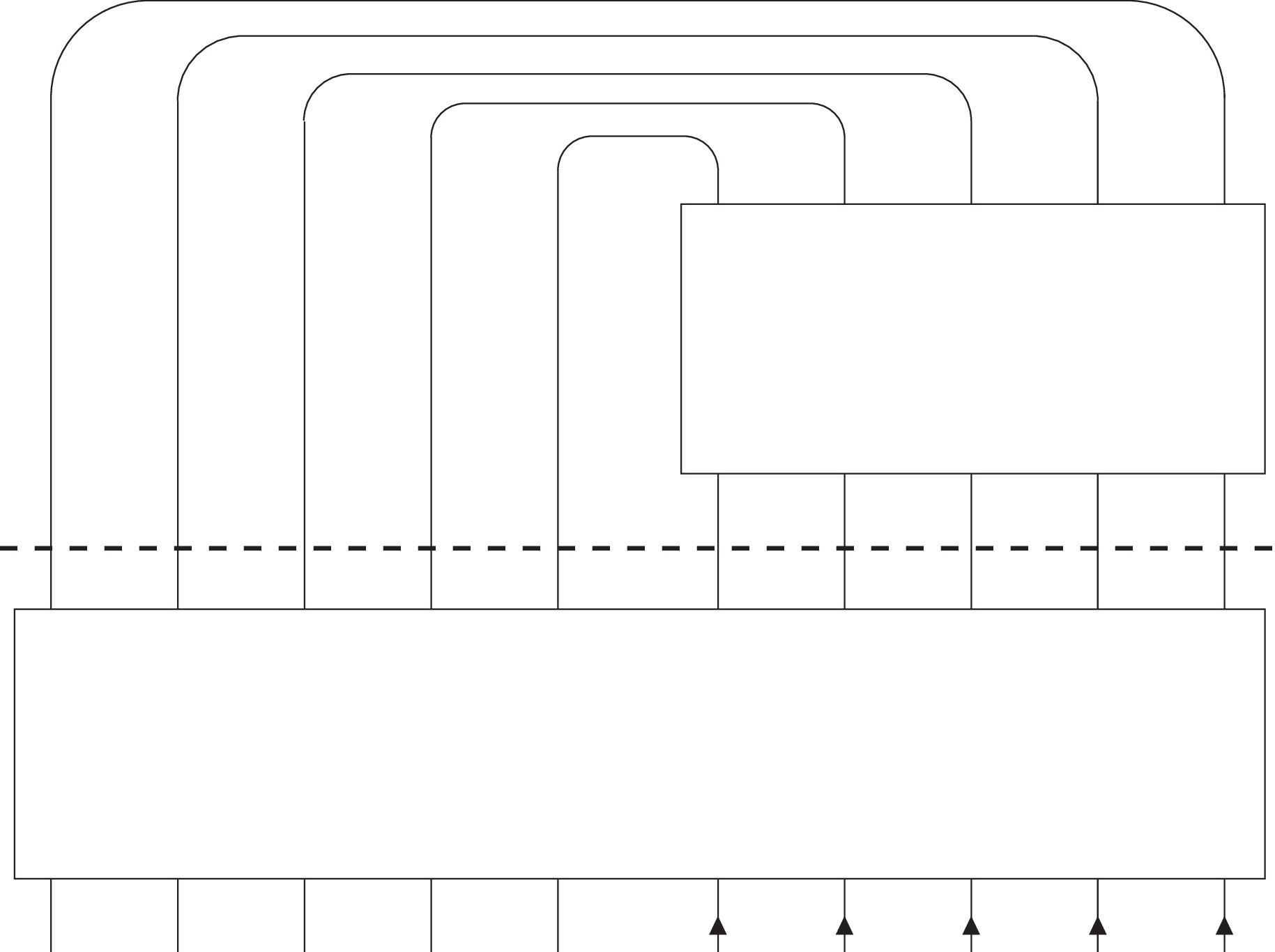}
\put(133,83){\large$\sigma$} \put(84,24){\large$\Sigma$}
\put(96,-12){$11$} \put(130,-11){$\cdots$} \put(165,-11){$ml_m$}
\end{overpic}} 
$$ 
\vspace{0.1cm}
\caption{A group action of $\mathcal{CH}(\overline{\bm l}{\bm l})$ on $\mathcal{CH}({\bm l})$.}\label{CLaction}
\end{figure}

\begin{figure}[h]
$$
\raisebox{10 pt}{\begin{overpic}[width=240pt]{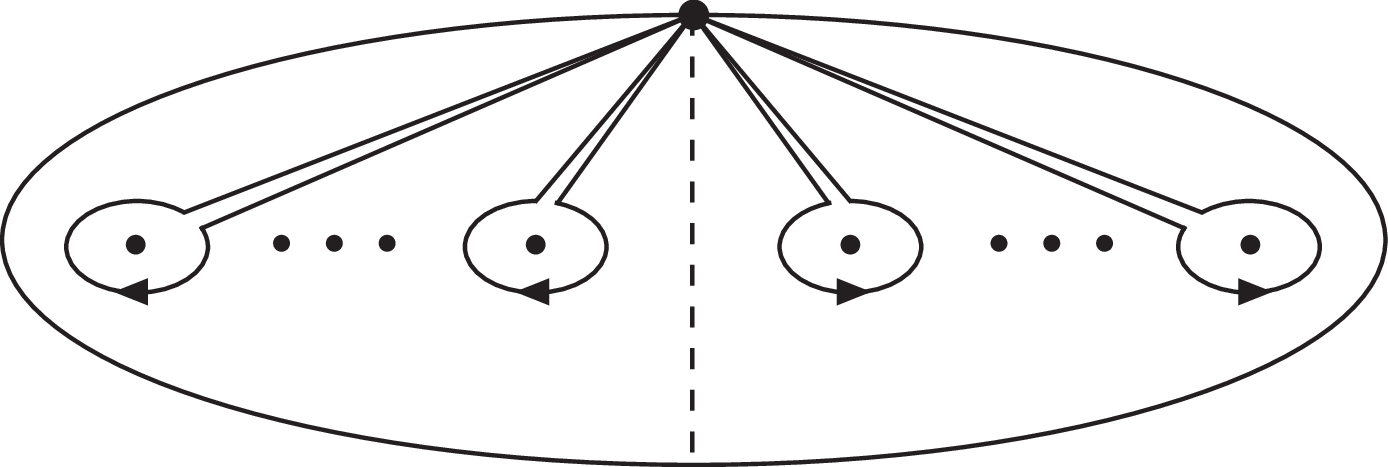}
\put(117,83){$b$} 
\put(87,15){$\overline{11}$} 
\put(28,16){$\overline{ml_m}$}
\put(141,17){$11$} 
\put(197,20){$ml_m$}
\end{overpic}} 
$$ 
\vspace{-1.0cm}
\caption{Generators of $\pi_1(D^2\setminus\{p_{\overline{ml_m}},\dots, p_{\overline{11}}, p_{11},\dots,p_{ml_m}\})$}\label{CLH2-generators}
\end{figure}

Let $\mathcal{S}_{CL}({\bm l})$ be a subset of $\mathcal{CH}(\overline{\bm l}{\bm l})$ whose elements $s$ satisfy $s\cdot {\bm 1}={\bm 1}$ for ${\bm 1} \in \mathcal{CH}({\bm l})$ up to CL-homotopy. Then $\mathcal{S}_{CL}({\bm l})$ is a subgroup of $\mathcal{CH}(\overline{\bm l}{\bm l})$. 
$(\overline{x}_{st}, \overline{x}_{st})_{ij}$) 
Consider an element $(x_{st}, x_{st})_{ij}$, 
$(\overline{x}_{st}, \overline{x}_{st})_{ij}$ and $(\overline{x}_{st}, x_{st})_{ij} \in \mathcal{S}_{CL}({\bm l})$ which are shown in Figure \ref{CLpartialconj}, \ref{barCLpartialconj} and \ref{CLconjugation}, respectively.
\par

\begin{figure}[h]
\vspace{0.4cm}
$$
\raisebox{-20 pt}{\begin{overpic}[width=183pt]{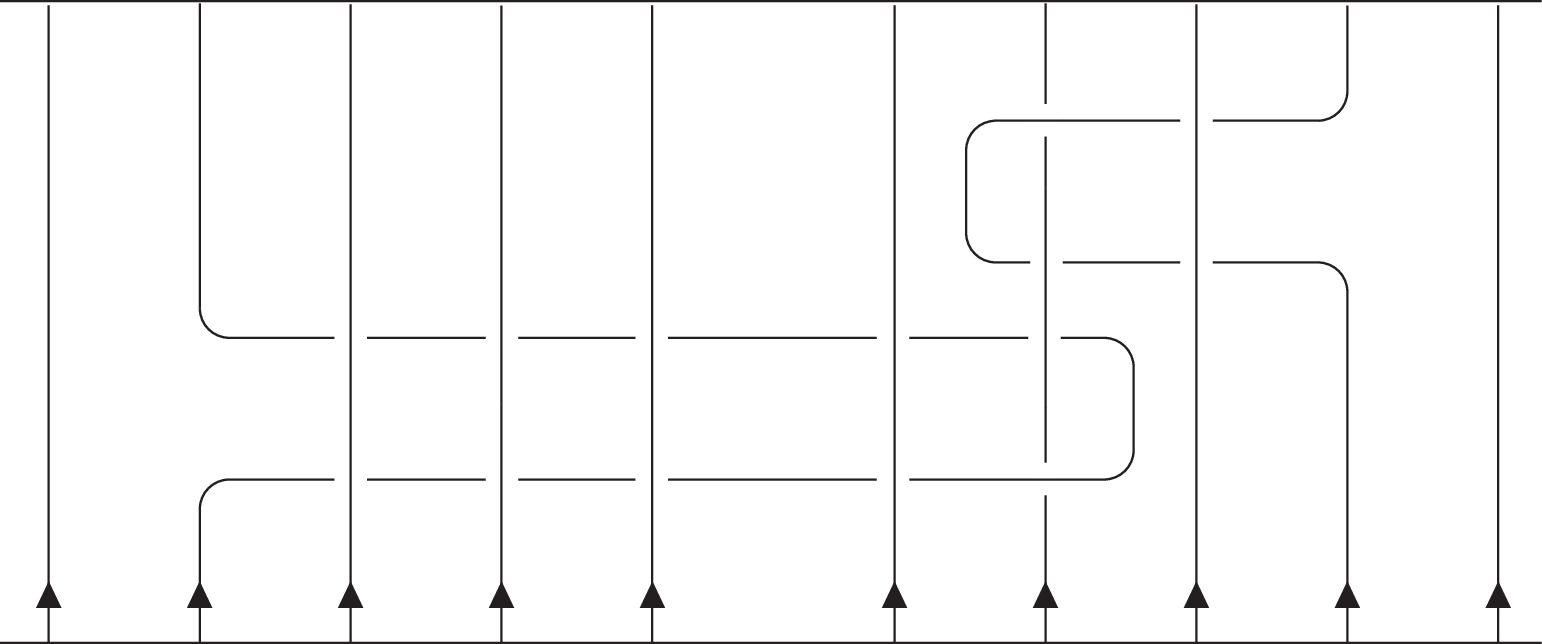}
\put(-2,-13){$\cdots$} \put(18,-14){$\overline{ij}$} \put(34,-13){$\cdots$}\put(54,-14){$\overline{st}$} \put(70,-13){$\cdots$}
\put(97,-13){$\cdots$} \put(118,-14){$st$} \put(135,-13){$\cdots$} \put(155,-14){$ij$} \put(170,-13){$\cdots$}
\put(79,-33){$i>s$}
\end{overpic}} 
\hspace{1.3cm}
\raisebox{-20 pt}{\begin{overpic}[width=183pt]{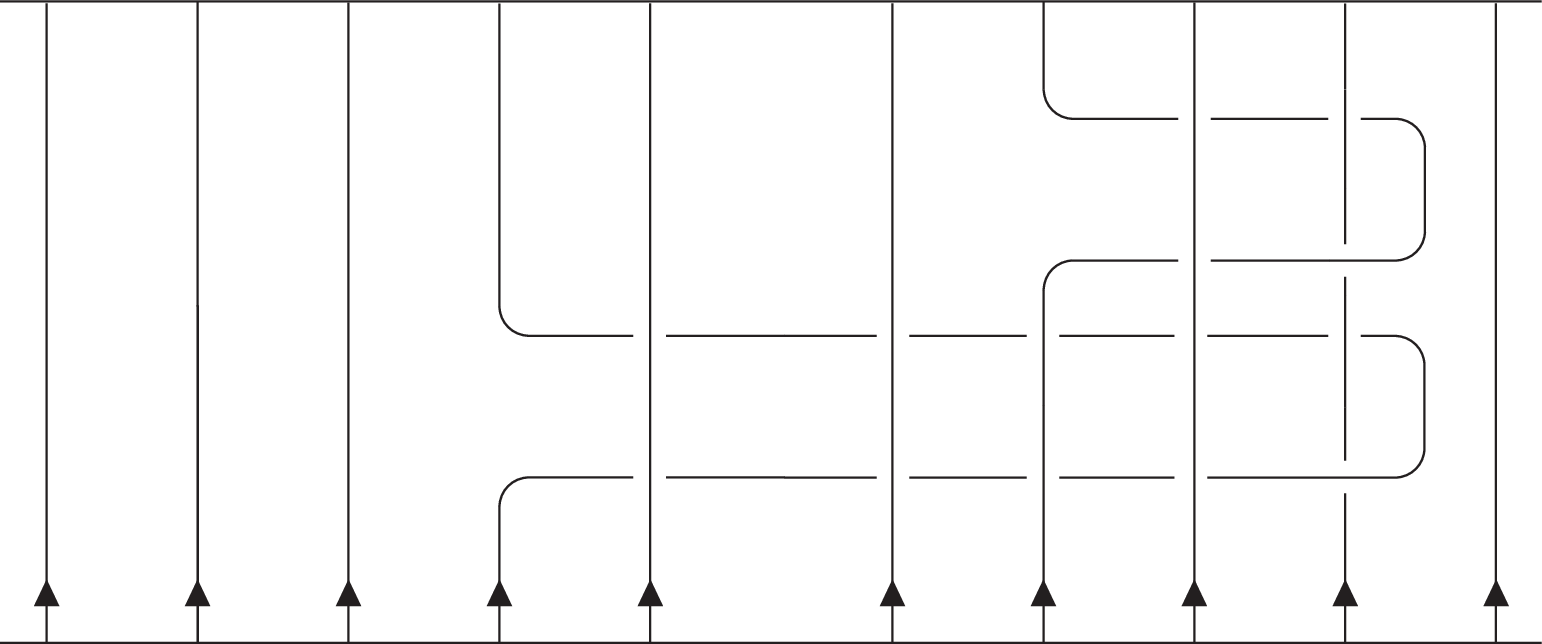}
\put(-2,-13){$\cdots$} \put(18,-14){$\overline{st}$} \put(34,-13){$\cdots$}\put(54,-14){$\overline{ij}$} \put(70,-13){$\cdots$}
\put(97,-13){$\cdots$} \put(118,-14){$ij$} \put(135,-13){$\cdots$} \put(155,-14){$st$} \put(170,-13){$\cdots$}
\put(79,-33){$i<s$}
\end{overpic}} 
$$
\vspace{0.8cm}
\caption{$(x_{st}, x_{st})_{ij}$.
}\label{CLpartialconj}
\end{figure}

\begin{figure}[h]
\vspace{0.4cm}
$$
\raisebox{-20 pt}{\begin{overpic}[width=183pt]{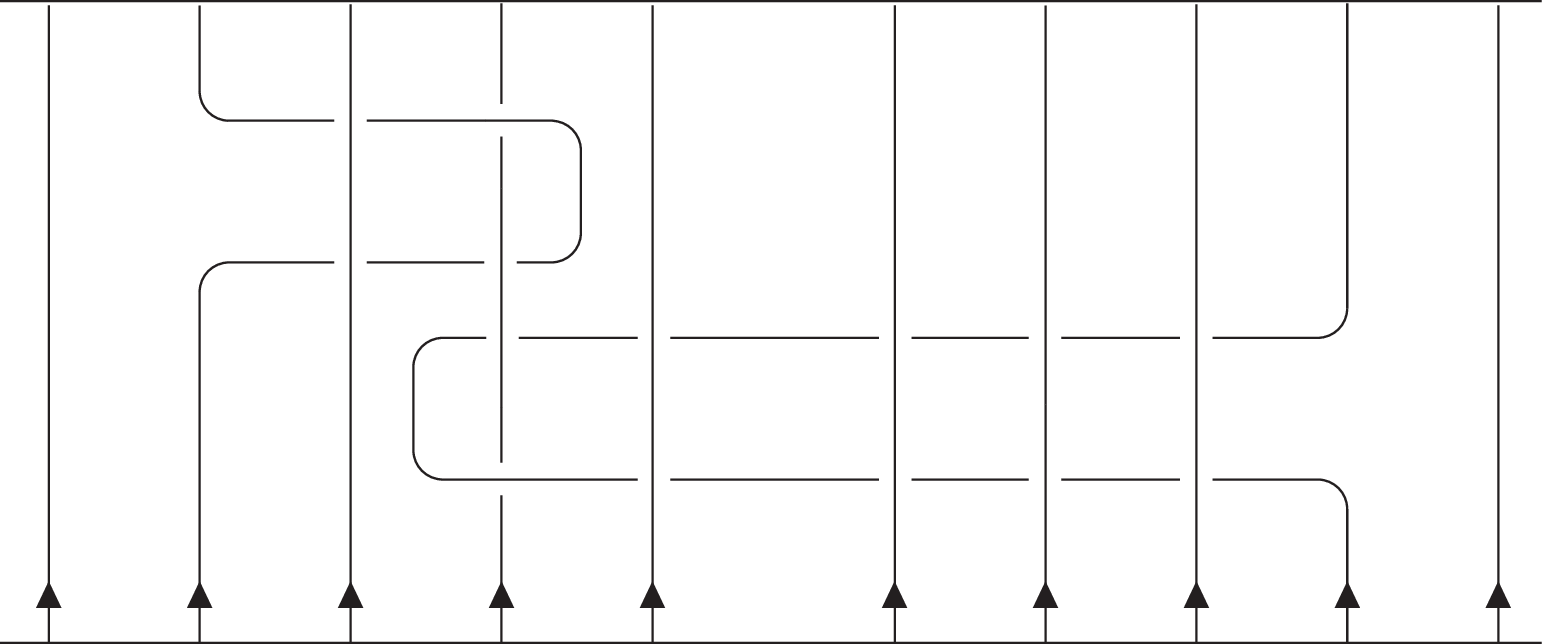}
\put(-2,-13){$\cdots$} \put(18,-14){$\overline{ij}$} \put(34,-13){$\cdots$}\put(54,-14){$\overline{st}$} \put(70,-13){$\cdots$}
\put(97,-13){$\cdots$} \put(118,-14){$st$} \put(135,-13){$\cdots$} \put(155,-14){$ij$} \put(170,-13){$\cdots$}
\put(79,-33){$i>s$}
\end{overpic}} 
\hspace{1.3cm}
\raisebox{-20 pt}{\begin{overpic}[width=183pt]{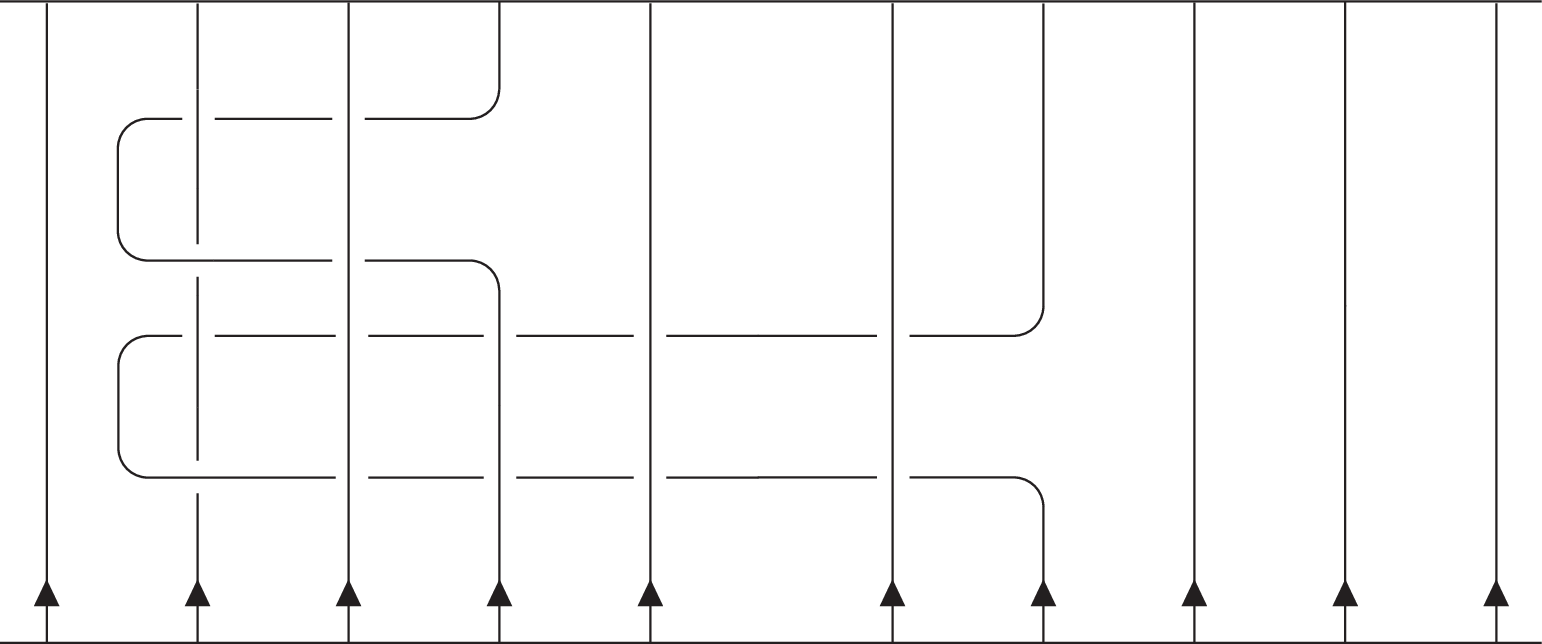}
\put(-2,-13){$\cdots$} \put(18,-14){$\overline{st}$} \put(34,-13){$\cdots$}\put(54,-14){$\overline{ij}$} \put(70,-13){$\cdots$}
\put(97,-13){$\cdots$} \put(118,-14){$ij$} \put(135,-13){$\cdots$} \put(155,-14){$st$} \put(170,-13){$\cdots$}
\put(79,-33){$i<s$}
\end{overpic}} 
$$ 
\vspace{0.8cm}
\caption{
$(\overline{x}_{st}, \overline{x}_{st})_{ij}$.
}\label{barCLpartialconj}
\end{figure}

\begin{figure}[h]
$$
\raisebox{-20 pt}{\begin{overpic}[width=183pt]{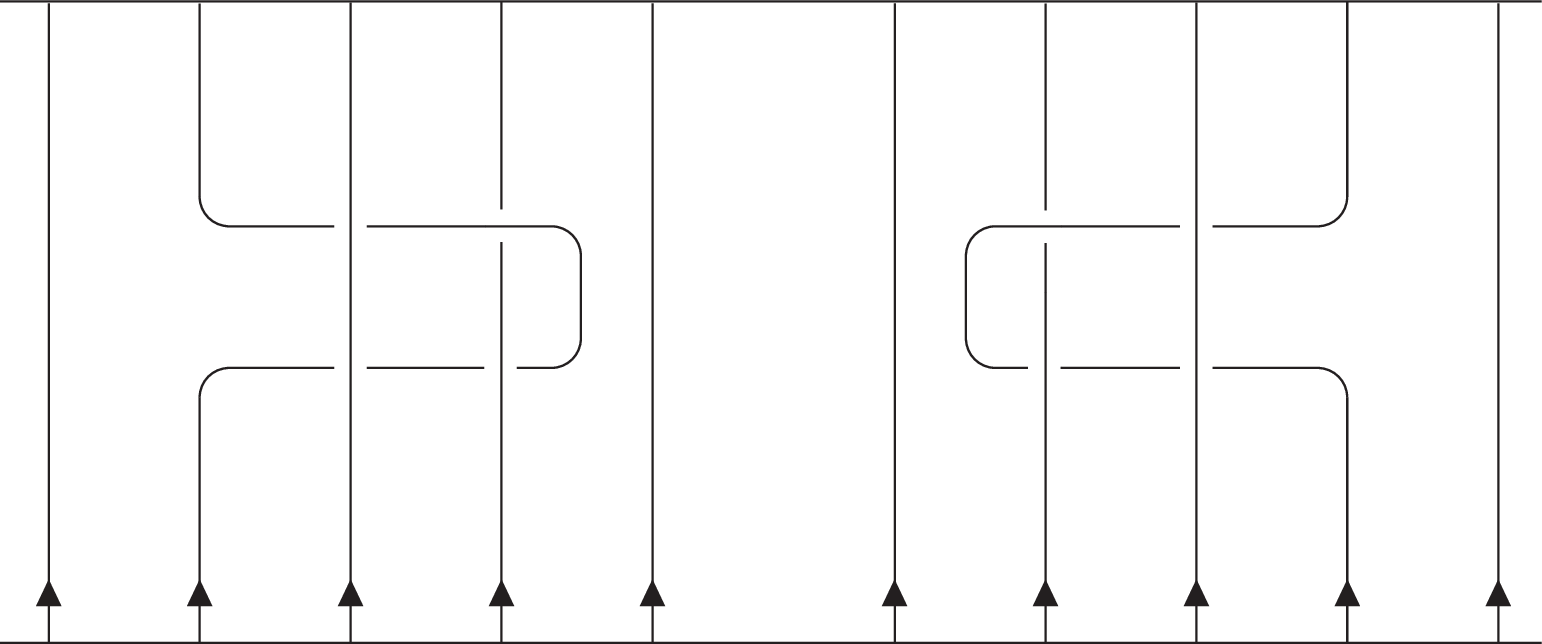}
\put(-2,-13){$\cdots$} \put(18,-14){$\overline{ij}$} \put(34,-13){$\cdots$}\put(54,-14){$\overline{st}$} \put(70,-13){$\cdots$}
\put(98,-13){$\cdots$} \put(118,-14){$st$} \put(135,-13){$\cdots$} \put(155,-14){$ij$} \put(170,-13){$\cdots$}
\end{overpic}}
$$ 
\vspace{0.1cm}
\caption{$(\overline{x}_{st}, x_{st})_{ij}$.}\label{CLconjugation}
\end{figure}

The following proposition is a corollary of Proposition \ref{L-kernel}, since the elements $(\overline{x}_{st}, \overline{x}_{st})_{ij}$ vanish up to CL-homotopy if $i=s$. 

\begin{proposition} \label{CL-kernel}
$\mathcal{S}_{CL}({\bm l})$ is generated by the elements $(\overline{x}_{st}, \overline{x}_{st})_{ij}$ such that $i\neq s$.
\end{proposition}

A \textit{d-base} for a colored link $L$ is a disc embedded in $S^3$ which intersects $L$ transversely at the $p_{ij}$ points of the disc with intersection numbers $+1$. 
A pair $(L,D)$ of a colored link $L$ and a d-base $D$ is called \textit{d-based colored link}. Two d-based colored links are \textit{CL-homotopic} if they are transformed to each other by a sequence of ambient isotopies and crossing changes between the same colored components, where the supporting balls of the crossing changes are disjoint from the d-bases. 
By taking the closure $\hat{\sigma}$ of a colored string link $\sigma$, 
$D\times \{0\}$ and $D\times \{1\}$ of the cylinder of $\sigma$ are identified and the disc becomes a canonical d-base $D$ of $\hat{\sigma}$ and we have d-based colored link $(\hat{\sigma}, D)$. 
An analogue of Corollary 2.4 of \cite{HL} we have the following proposition. 

\begin{proposition}
The correspondence $\sigma\mapsto (\hat{\sigma}, D)$ induces a bijection between the set of CL-homotopy classes of colored string links and that of b-based colored links. 
\end{proposition}

Then we have an Markov-type theorem for colored links. 
\begin{theorem}\label{Markov-type-for-CL}
Let $\sigma$ and $\sigma'$ be in $\mathcal{CH}({\bm l})$. Then the closures of $\sigma$ and $\sigma'$ are CL-homotopic if and only if there is an element $s\in\mathcal{S}_{CL}({\bm l})$ such that $s\cdot\sigma=\sigma'$. 
\end{theorem}

\begin{proof}
The proof is done parallel to that of Theorem 2.9 in \cite{HL} (Theorem \ref{Lclassify} above) by using the d-bases for colored links. 
\end{proof}

\subsection{Milnor invariants for colored string links and colored links}
The Milnor invariants for (string) links are also available for colored (string) links by using the component numbers $(i,j)$ for colored (string) links. For sequences with non-repeating colors, we have CL-homotopy invariants. 


\begin{proposition}\label{Milnorinvarinat}
For a colored link $L$ and a sequence of component numbers
$$I=(i_0, j_0)(i_1,j_1)\cdots(i_k,j_k) $$
with no-repeating colors (i.e. $i_s\neq i_t$ if $s\neq t$), 
the Milnor invariant $\overline{\mu}_{L}(I)$ is a CL-homotopy invariant of $L$. 
\end{proposition}

\begin{proof}
Let  $L$ be an $m$-colored link with a component decomposition ${\bm l}=(1_{l_1}, \cdots, m_{l_m})$.
Let $G_L=\pi_1({S}^3\setminus{L})$ and $c=\sum_{i=1}^{m} l_i$. Then $G_L / \Gamma_{c+1}G_L$ has a presentation 
$$\left< \alpha_{ij}  \mid [\alpha_{ij}, \beta_{ij}] \ (1 \leq i \leq m, 1 \leq j \leq  l_i), \Gamma_{c+1}A \right>,$$
where $\alpha_{ij}$ and $\beta_{ij}$ are a meridian and a longitude of $(i,j)$-th component respectively, $A$ is a free group generated by $\alpha_{ij} \ (1 \leq i \leq m, 1 \leq j \leq  l_i)$ and $\Gamma_{c+1}$ is the ($c+1$)-th lower central series.
Moreover $G_L / \prod_{i=1}^{m} (A_i)_2$ is a quotient group of $G_L / \Gamma_{c+1}G_L$ which has a presentation 
$$\left< \alpha_{ij}   \mid [\alpha_{ij}, \beta_{ij}], (A_i)_2 \ (1 \leq i \leq m, 1 \leq j \leq  l_i) \right>, $$ 
where $A_i$ is the normal subgroup of $G_L$ generated by the meridians $\alpha_{i 1},\cdots \alpha_{i l_i}$ and $(A_i)_2$ is the commutator subgroup of $A_i$.
We denote the quotient group by $\mathcal{CG}_L$.
As in the case of link-homotopy for links, the quotient group $\mathcal{CG}_L$ is invariant under CL-homotopy. Furthermore, the conjugate classes $\lambda \alpha_{ij} \lambda^{-1}$ of the meridian $\alpha_{ij}$ of the $(i,j)$-th component in $\mathcal{CG}_L$ and $\lambda \beta_{ij} \lambda^{-1}$ of the longitude $\beta_{ij}$ of the $(i,j)$-th component in $\mathcal{CG}_L/A_i$ are invariant under CL-homotopy.
To prove that $\overline{\mu}_L ((i_0, j_0)(i_1,j_1)\cdots(i_k,j_k))$ is a CL-homotopy invariant, it is sufficient to prove that it is not changed when 
\begin{itemize}
\item[(1)] the word $w_{i_k j_k}$ for the longitude $\beta_{i_k j_k}$ is replaced by a conjugation, 
\item[(2)]some $\alpha_{i j}$ is replaced by a conjugate, 
\item[(3)] $w_{i_k j_k}$ is multiplied by a product of conjugates of $\alpha_{ij}^{-1} w_{ij}^{-1} \alpha_{ij} w_{ij}$, 
\item[(4)] $w_{i_k j_k}$ is multiplied by an element of $(A_i)_2$, 
\item[(5)] $w_{i_k j_k}$ is multiplied by an element of $A_{i_k}$. 
\end{itemize}
The invariance for (1), (2), (3) is proved in \cite{Mil2}.
To prove (4), each term of the image of an element of $(A_i)_2$ by the Magnus expansion contains any of the factors $X_{i1}, \cdots ,X_{il_i}$ at least twice, except for the constant term 1. But $\mu_{L}(I)$ is the coefficient of a term in the Magnus expansion of $w_{i_k j_k}$ which contains any of the factors $X_{i1}, \cdots ,X_{il_i}$ at most one.   
Similarly, to prove (5), each term of the image of an element of $(A_{i_k})$ by the Magnus expansion contains any of the factors $X_{i_k 1}, \cdots , X_{i_k l_{i_k}}$, except for the term 1.  But $\mu_{L}(I)$ is the coefficient of a term in the Magnus expansion of $w_{i_k j_k}$ which does not involve $X_{i_k 1}, \cdots ,X_{i_k l_{i_k}}$ as a factor.   
\end{proof}

In the case of colored string links, we have the following proposition.

\begin{proposition}
For a colored string link $\sigma$ and a sequence of component numbers
$$I=(i_0, j_0)(i_1,j_1)\cdots(i_k,j_k) $$
with no-repeating colors (i.e. $i_s\neq i_t$ if $s\neq t$), 
the Milnor invariant $\mu_{\sigma}(I)$ is a CL-homotopy invariant of $\sigma$. 
\end{proposition}

\begin{proof}
The proof is similar to the link version. 
Let $\sigma$ be an $m$-colored link with a component decomposition ${\bm l}=(1_{l_1}, \cdots, m_{l_m})$ and $G_\sigma$ the fundemental group of the complement of $\sigma$. In this case,
we can fix the meridian and the preferred longitude and by Stallings' theorem \cite{Sta} $G_\sigma / \Gamma_{c+1}G_\sigma$ has a presentation 
$$\left< \alpha_{ij} \ (1 \leq i \leq m, 1 \leq j \leq  l_i) \mid \Gamma_{c+1}A \right>.$$
Moreover, $G_\sigma / \prod_{i=1}^{m} (A_i)_2$ has a presentation 
$$\left< \alpha_{ij} \ (1 \leq i \leq m, 1 \leq j \leq  l_i) \mid (A_i)_2 \ (1 \leq i \leq m)  \right>.$$ 
Therefore, in order to prove that ${\mu}_{\sigma} ((i_0, j_0)(i_1,j_1)\cdots(i_k,j_k))$ is CL-homotopy invariant, it is sufficient to prove that it is not changed when 
\begin{itemize}
\item[(4')] $w_{i_k j_k}$ is multiplied by an element of $(A_i)_2$, 
\item[(5')] $w_{i_k j_k}$ is multiplied by an element of $A_{i_k}$. 
\end{itemize}
The invariance for (4') and (5') is proved in the proof of Proposition \ref{Milnorinvarinat}.
\end{proof}

We call $\mu_{\sigma}(I)$ (resp. $\overline{\mu}_{L}(I)$) for a sequence $I$ with non-repeating colors a \textit{Milnor homotopy invariants} for colored string links (resp. colored links). 

\begin{theorem}
For a component decomposition ${\bm l}$, the set $\mathcal{CH}({\bm l})$ of CL-homotopy classes of colored string links are classified by 
the Milnor homotopy invariants $\mu_{\sigma}(I)$. 
\end{theorem}

\begin{proof}
Theorem 4.3 in \cite{Y} (see also Theorem 4.2 in \cite{MY2}) showed that a link-homotopy class of string link $S$ has a canonical form up to link-homotopy which consists of the trivial string link and claspers. For details of the clasper, see \cite{Ha}. The canonical form is determined by the Milnor homotopy invariants $\mu_S(I)$ for $S$. This gives an alternative proof for the fact that the set of link-homotopy classes is classified by $\mu_S(I)$ (\cite{HL}). 
\par
Using the similar techniques in \cite{Y}, we have a canonical form of a CL-homotopy class of a colored string link $\sigma$. 
For simplicity, we use the canonical form used in \cite{MY2} which is slightly modified from the one in \cite{Y}. 
Then a canonical form for colored string link up to CL-homotopy is obtained from the one of Theorem 4.2 in \cite{MY2} by omitting the claspers which have two leaves attaching the component with the same colors. The canonical form is determined by the Milnor homotopy invariants $\mu_{\sigma}(I)$ for colored string links. Namely, $\sigma$ is CL-homotopic to $\Sigma=\Sigma_1\Sigma_2\cdots\Sigma_{m-1}$, where
$$\Sigma_i=\prod_{J\in \mathcal{J}_i}T_J^{x_J} \mbox{ and } x_J=\mu_{\Sigma_i}(J)
=\left\{\begin{array}{lc}
\mu_{\sigma}(J) & (i=1),\\[2pt]
\mu_{\sigma}(J)-\mu_{\Sigma_1\Sigma_2\cdots\Sigma_{i-1}}(J) & (i>1).
\end{array}\right.$$
Here $\mathcal{J}_k$ is the set of all sequences $(i_0,j_0)(i_1,j_1)\cdots(i_k,j_k)$ of distinct component numbers with length $k+1$ satisfying $(i_0,j_0)<(i_l,j_l)<(i_k,j_k)$ if $l\neq 0, k$. 
$T^{\pm 1}_J$ is a string link shown in Figure \ref{clasper-surgery}, where we identify the string link attached by the clasper with the string link obtained by the surgery along the clasper.  
\par
This result induces that $\mathcal{CH}({\bm l})$ is classified by the Milnor homotopy invariants for colored string links. 
\end{proof}

\begin{figure}[h]
$$
\raisebox{-20 pt}{\begin{overpic}[width=160pt]{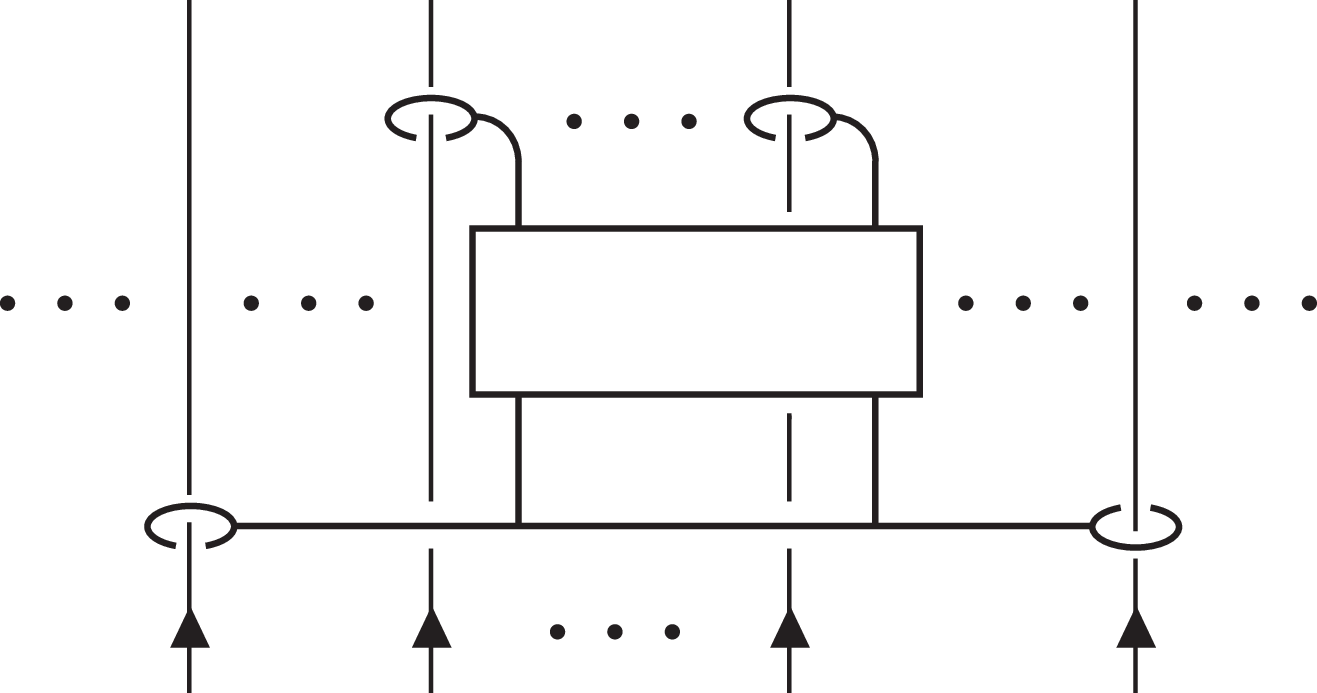}
\put(80,42){$b_J$}
\put(14,-12){$i_0j_0$} 
\put(44,-12){$i_1j_1$}
\put(76,-12){$i_{k-1}j_{k-1}$} 
\put(129,-12){$i_kj_k$} 
\end{overpic}}
\hspace{1.6cm}
\raisebox{-20 pt}{\begin{overpic}[width=160pt]{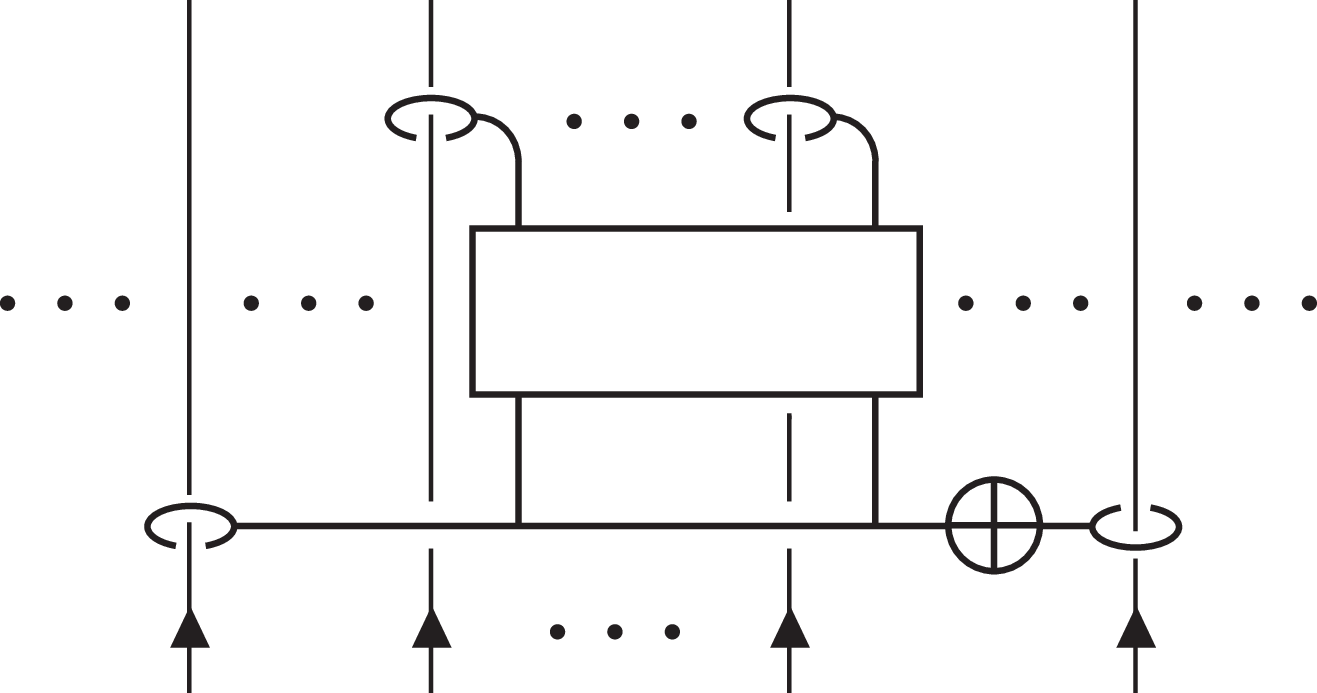}
\put(80,42){$b_J$}
\put(14,-12){$i_0j_0$} 
\put(44,-12){$i_1j_1$}
\put(76,-12){$i_{k-1}j_{k-1}$} 
\put(129,-12){$i_kj_k$} 
\end{overpic}}
$$ 
\vspace{0.1cm}
\caption{$T_J$ and $T_J^{-1}$.}\label{clasper-surgery}
\end{figure}
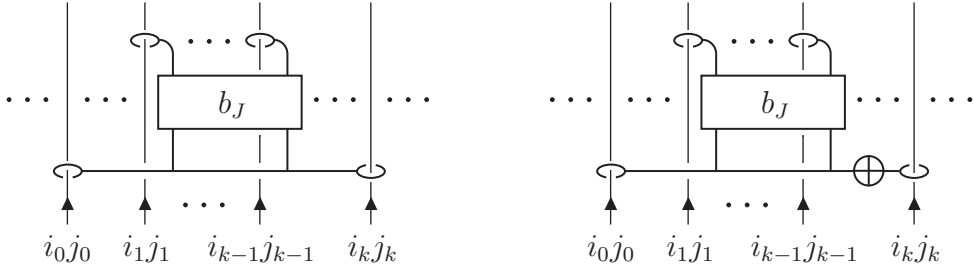

\subsection{A classification of the component-homotopy classes of spatial graphs}
For a graph $G$ with $m$ components, we take a family of spanning trees $T=(t_1, \dots, t_m)$, where $t_i$ is a spanning tree of the $i$-th component of $G$.
For each component of a spatial graph $\psi(G)$, we obtain a spatial B-graph by contracting a subgraph $\psi(t_i)$ to a point for each component. We denote the spatial B-graph by $\psi(G)_T$ (see Figure \ref{contraction}). 

\begin{figure}[h]
$$
\raisebox{-33 pt}{\begin{overpic}[width=100pt]{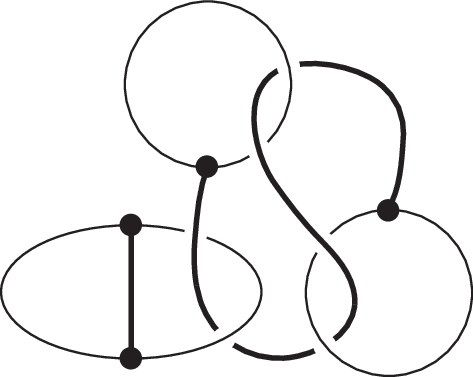}
\put(-7,50){$\psi(G)$} 
\put(6,15){\footnotesize $\psi(t_1\!)$} 
\put(82,64){\footnotesize $\psi(t_2)$}
\end{overpic}}
\hspace{0.2cm}\longrightarrow\hspace{0.2cm}
\raisebox{-33 pt}{\begin{overpic}[width=100pt]{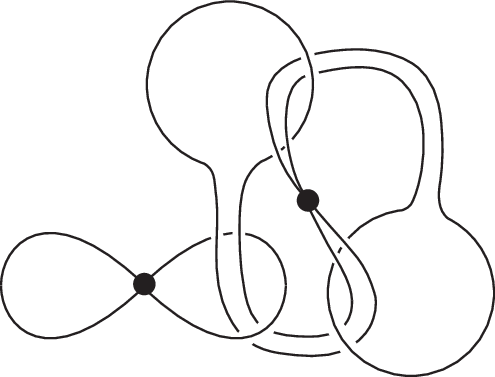}
\put(-7,50){$\psi(G)_T$} 
\end{overpic}}
$$ 
\vspace{0.0cm}
\caption{A contraction of $\psi(G)$ to bouquet graphs $\psi(G)_T$.}\label{contraction} 
\end{figure}

\begin{proposition} \label{bouquet}
For a graph $G$ and its spanning trees $T$, two spatial graphs $\psi_1(G)$ and $\psi_2(G)$
are component-homotopic if and only if $\psi_1(G)_T$ and $\psi_2(G)_T$ are component-homotopic. 
\end{proposition}

\begin{proof}
Assume that $\psi_1(G)_T$ and $\psi_2(G)_T$ are component-homotopic. 
Then there is a sequence $B_i$ ($1\leq i\leq k$) of B-graphs, where $B_1=\psi_1(G)_T$, $B_k=\psi_2(G)_T$ and $B_{i+1}$ is obtained from $B_i$ by a self-crossing change. For each $B_i$, by expanding the vertex to an embedding of $T$, we have a spatial graph $B_{i}^T$ which is an embedding of $G$. Note that $B_{i}^T$ is uniquely determined from $B_{i}$ and $T$ up to component-homotopy. The sequence $B^T_i$ ($1\leq i\leq k$) satisfies that $B_1^T=\psi_1(G)$, $B_k^T=\psi_2(G)$ and $B_{i+1}^T$ is component-homotopic to $B_{i}^T$. Then $\psi_1(G)$ and $\psi_2(G)$ are component-homotopic. 
\par
On the other hand, if $\psi_1(G)$ and $\psi_2(G)$ are component-homotopic, there is a sequence $\phi_i(G)$ ($1\leq i\leq k$) of spatial embeddings of $G$, where $\phi_1(G)=\psi_1(G)$, $\phi_k(G)=\psi_2(G)$ and $\phi_{i+1}(G)$ is obtained from $\phi_{i}(G)$ by a self-crossing change.  
By contracting $\phi_i(T)$ to a vertex for each $i$, we obtain a sequence $\phi_i(G)_T$ of B-graphs satisfying $\phi_1(G)_T=\psi_1(G)_T$ and $\phi_k(G)_T=\psi_2(G)_T$. 
Let $\beta_i$ be a supporting ball of the self-crossing change from $\phi_i(G)$ to $\phi_{i+1}(G)$. Assume the crossing change is done between edges $\phi_i(e)$ and $\phi_i(e')$.
If both $\phi_i(e)$ and $\phi_i(e')$ are not the edges of $\phi_i(T)$, after the contraction of $\phi_i(T)$, $\phi_{i+1}(G)_T$ is obtained from $\phi_{i}(G)_T$ by a self-crossing change at $\beta_i$. If $\phi_i(e)$ belongs to $\phi_i(T)$ and $\phi_i(e')$ do not, by the contraction of $\phi_i(T)$, 
the arc of $\phi_i(e)$ in $\beta_i$ changes to parallel arcs of edges of $\phi_i(G)_T$, see Figure \ref{self-crossing}. 
Then $\phi_{i+1}(G)_T$ is obtained from $\phi_{i}(G)_T$ by several self-crossing changes at $\beta_i$. 

\begin{figure}[h]
$$
\raisebox{-44 pt}{\begin{overpic}[width=60pt]{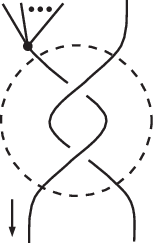}
\put(60,64){$\beta_i$} \put(17,-16){$\phi_i(G)$}
\put(42,45){\colorbox{white}{\footnotesize$\phi_i(e)$}} 
\put(-11.5,45){\colorbox{white}{\footnotesize$\phi_i(e')$}} 
\put(-57,8){contraction}
\end{overpic}} 
\hspace{0.7cm}
\longrightarrow
\hspace{0.7cm}
\raisebox{-44 pt}{\begin{overpic}[width=60pt]{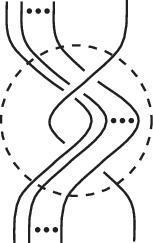}
\put(60,64){$\beta_i$} \put(16,-16){$\phi_i(G)_T$}
\end{overpic}} 
\hspace{0.7cm}
\longrightarrow
\hspace{0.7cm}
\raisebox{-44 pt}{\begin{overpic}[width=60pt]{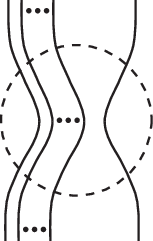}
\put(60,64){$\beta_i$} \put(11,-16){$\phi_{i+1}(G)_T$}
\end{overpic}} 
$$ 
\vspace{0.1cm}
\caption{The supporting ball $\beta_i$ and the contraction of $\phi_i(T)$.}\label{self-crossing}
\end{figure}
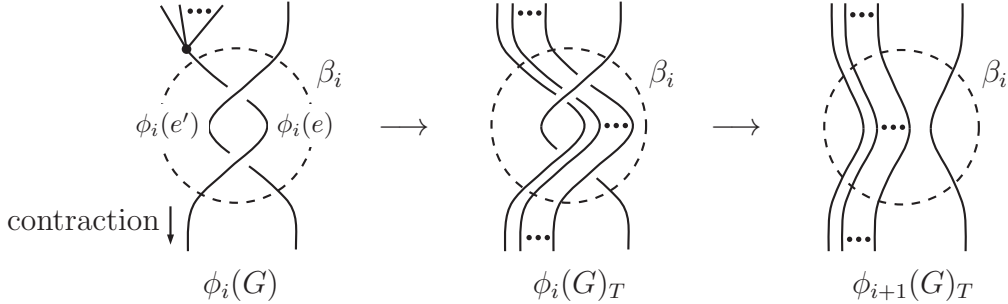
For the other cases where only $\phi_i(e')$ belongs to $\phi_i(T)$ and both $\phi_i(e)$ and $\phi_i(e')$ belong to $\phi_i(T)$, similar to the above case, we have that $\phi_{i+1}(G)_T$ is component-homotopic to $\phi_{i}(G)_T$. 
Thus $\psi_1(G)_T$ is component-homotopic to $\psi_2(G)_T$. 
\end{proof}

For a graph $G$, if we fix a family $T$ of spanning trees of components, 
by Proposition \ref{bouquet}, the classification problem of the component-homotopy classes of spatial embeddings $\psi(G)$ of a graph $G$ 
is reduced to that of the component-homotopy classes of spatial B-graphs $\psi(G)_T$. So we focus to classify $\mathcal{G}({\bm l})$ which is the set of component homotopy classes of spatial B-graphs. 

\begin{definition}
For $\sigma \in \mathcal{CH}({\bm l})$, a \textit{G-closure} $\hat{\sigma}$ of $\sigma$ is a spatial B-graph obtained by connecting the endpoints of the sub-string link $\sigma_i$ colored by $i$ to a vertex for each $i$
as in Figure \ref{G-closure-fig}. 
\end{definition}

\begin{figure}[h]
$$
\raisebox{-20 pt}{\begin{overpic}[width=187pt]{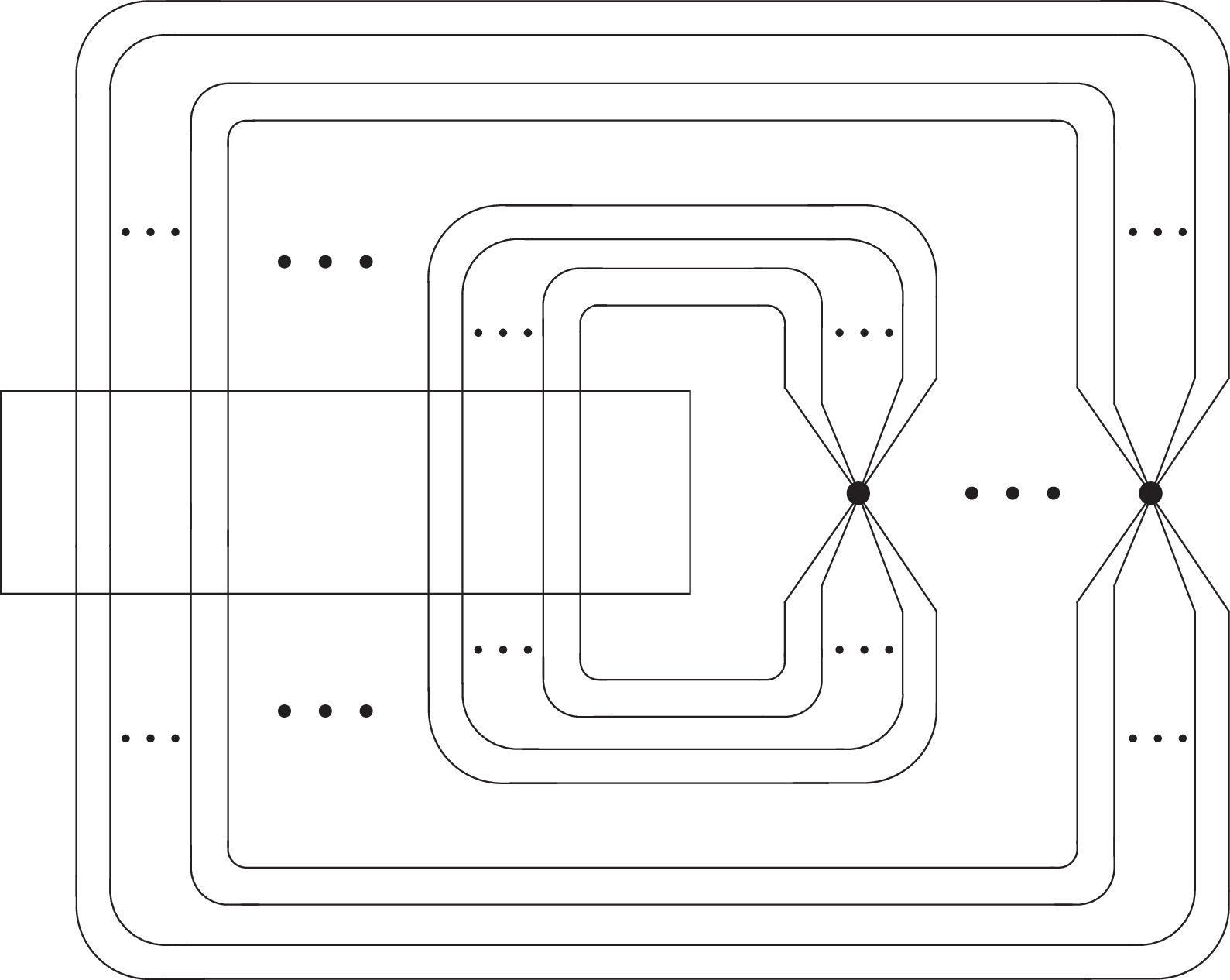}
\put(-15,70){\Large$\sigma$} 
\put(37,72){$\sigma_1$} 
\put(51,72){\tiny$\cdots$}
\put(89.5,72){$\sigma_m$} 
\end{overpic}}
$$ 
\caption{The G-closure $\hat{\sigma}$ of $\sigma$.}\label{G-closure-fig}
\end{figure}
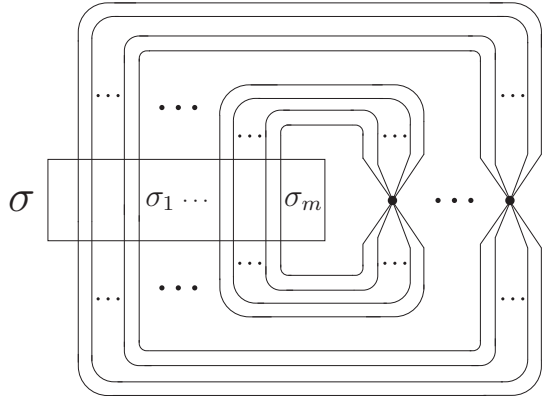

\begin{definition}
For ${\bm l}=(1_{l_1},\dots,m_{l_m})$, let $G_0=G_0({\bm l})$ be a proper embedding of star graphs each of which has $2l_i$ edges to $D\times I$ whose endpoints are at the fixed points $p_{ij}$ in $D\times\{0\}$ and $D\times\{1\}$ as in Figure \ref{g-zero} left. For a colored string link $\Sigma$ with the component decomposition $\overline{\bm l}{\bm l}$, a \textit{$G_0$-cap} $\Sigma\cdot G_0$ of $\Sigma$ is the composition of $\Sigma$ and $G_0$ (Figure \ref{g-zero} right). 
\end{definition}

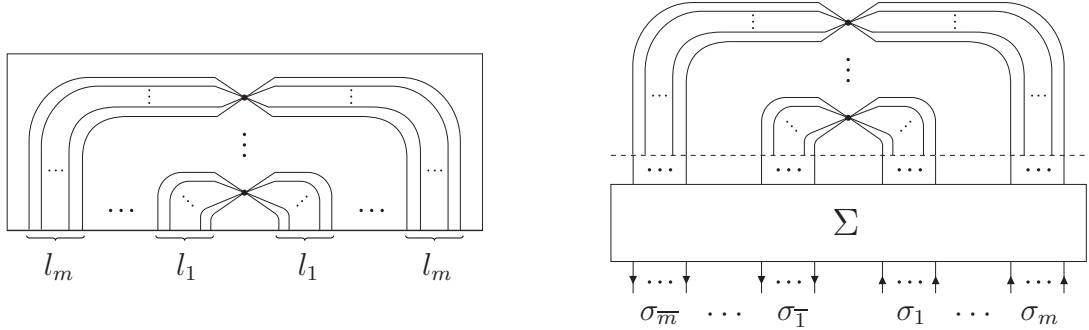
\begin{figure}[h]
$$
\raisebox{0 pt}{\begin{overpic}[width=180pt]{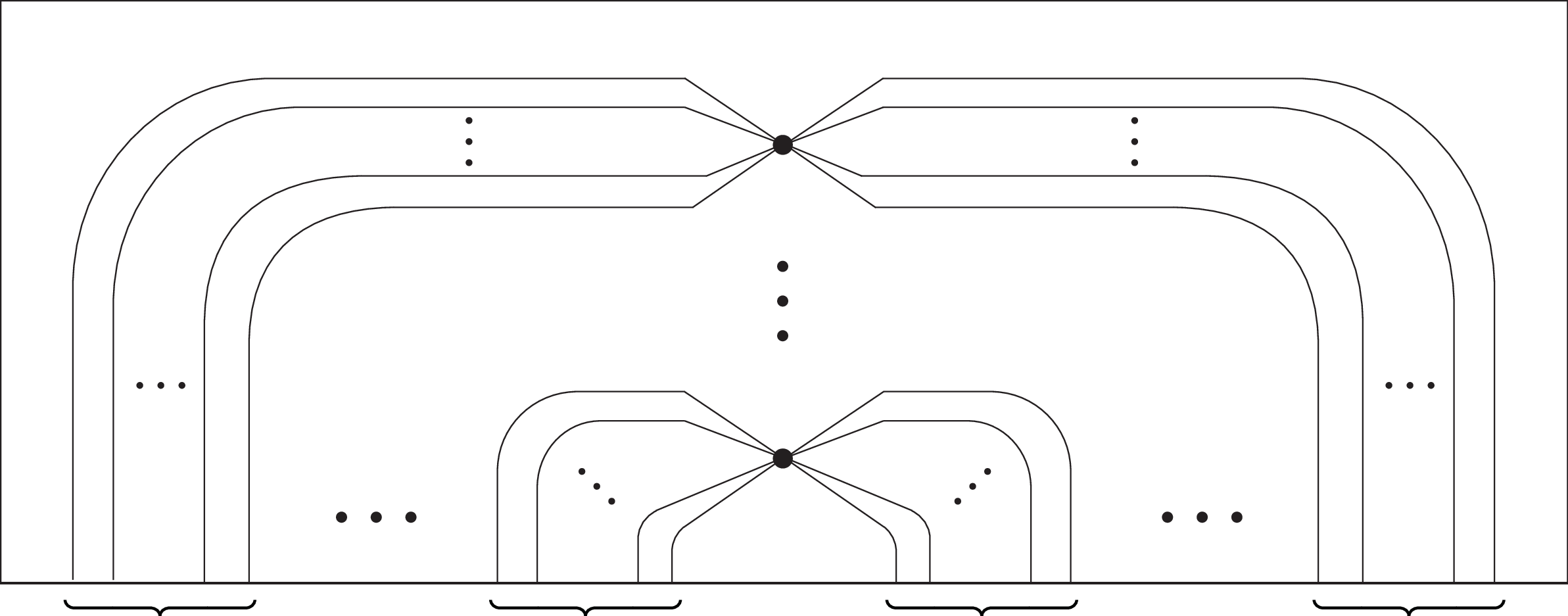}
\put(64,-13){$l_1$} 
\put(14,-13){$l_m$}
\put(110,-13){$l_1$} 
\put(157,-13){$l_m$}
\end{overpic}} 
\hspace{1.7cm}
\raisebox{-20 pt}{\begin{overpic}[width=180pt]{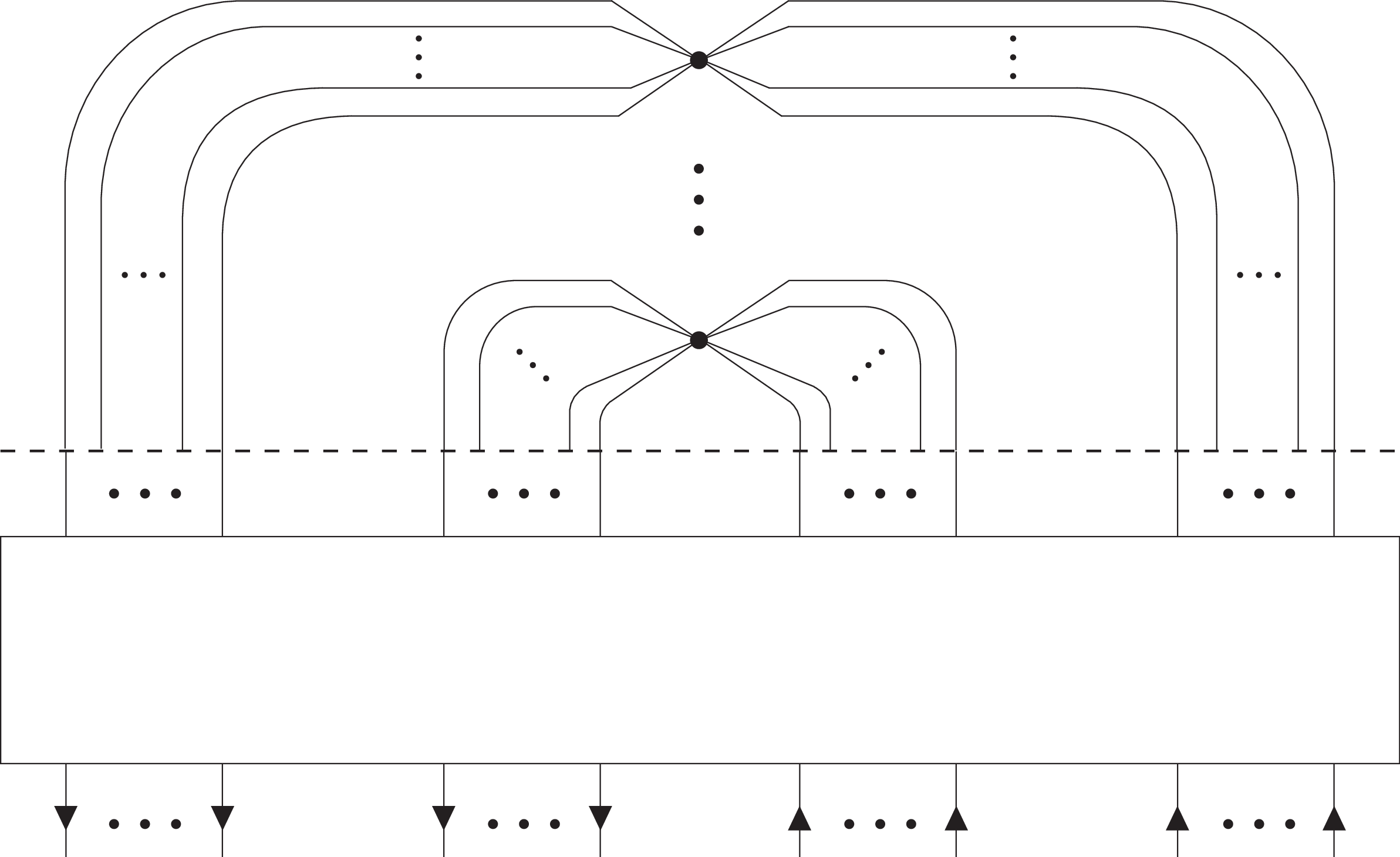}
\put(84,22){\large$\Sigma$}
\put(11,-10){$\sigma_{\overline{m}}$} \put(36,-11){$\cdots$} \put(63,-10){$\sigma_{\overline{1}}$}
\put(108,-10){$\sigma_1$} \put(130,-11){$\cdots$} \put(155,-10){$\sigma_{m}$}
\end{overpic}} 
$$ 
\vspace{0.1cm}
\caption{$G_0$ and $G_0$-cap of $\Sigma$.}\label{g-zero}
\end{figure}

The fundamental group $\pi_1(D^2\setminus G_0)$ has a presentation
$$\pi_1(D^2\setminus G_0)=\{x_{ij}, \overline{x}_{ij}\, ( 1\leq i\leq m, 1\leq j\leq l_i) \mid (\overline{X}_i)^{-1}X_i=1 \,(1\leq i\leq m)\}$$
where $X_i=\prod_{j=1}^{l_i}x_{ij}$ and $\overline{X}_i=\prod_{j=1}^{l_i}\overline{x}_{ij}$.

\par
Let $\mathcal{S}_{G}({\bm l})$ be a subset of $\mathcal{CH}(\overline{\bm l}{\bm l})$ whose elements $\Sigma$ satisfy $\Sigma\cdot G_0=G_0$ up to component-homotopy relative to endpoints. Then $\mathcal{S}_{G}({\bm l})$ is a subgroup of $\mathcal{CH}(\overline{\bm l}{\bm l})$. 

\begin{definition}
For a spatial B-graph $\Gamma$, a ball $b$ satisfying that $(b,b\cap \Gamma)$ is isotopic to $(D\times I, G_0)$ is called a \textit{b-base} of $\Gamma$. A pair $(\Gamma, b)$ of a spatial B-graph $\Gamma$ and its b-base $b$ is called \textit{b-based} spatial B-graph. Two b-based spatial B-graphs $(\Gamma, b)$ and $(\Gamma', b')$ are \textit{component-homotopic} if there is a homotopy from $(\Gamma, b)$ to $(\Gamma', b')$ which changes $b$ to $b'$ isotopically and $\Gamma$ to $\Gamma'$ as a component-homotopy with the supporting balls of self-crossing changes disjoint from the images of $b$. 
\end{definition}

We easily have the following lemma. 
\begin{lemma}
For any spatial B-graph $\Gamma$, there is a colored string link $\sigma$ such that the G-closure $\widehat{\sigma}$ is ambient isotopic to $\Gamma$.
\end{lemma}

For a b-based spatial B-graph $(\Gamma,b)$, we can transform the b-base $b$ to $D\times I$ (Figure \ref{b-disc} middle). 
By contracting the length of $I$ too small and treating $D\times I$ as $D$, we can treat the b-based spatial B-graph $(\Gamma, b)$ as a b-based colored link $(L,D)$, where $L$ is a link obtained from $\Gamma\setminus b$  by identifying the boundaries of each component, see Figure \ref{b-disc} right. 


\begin{figure}[h]
$$
\raisebox{-50 pt}{\begin{overpic}[width=120pt]{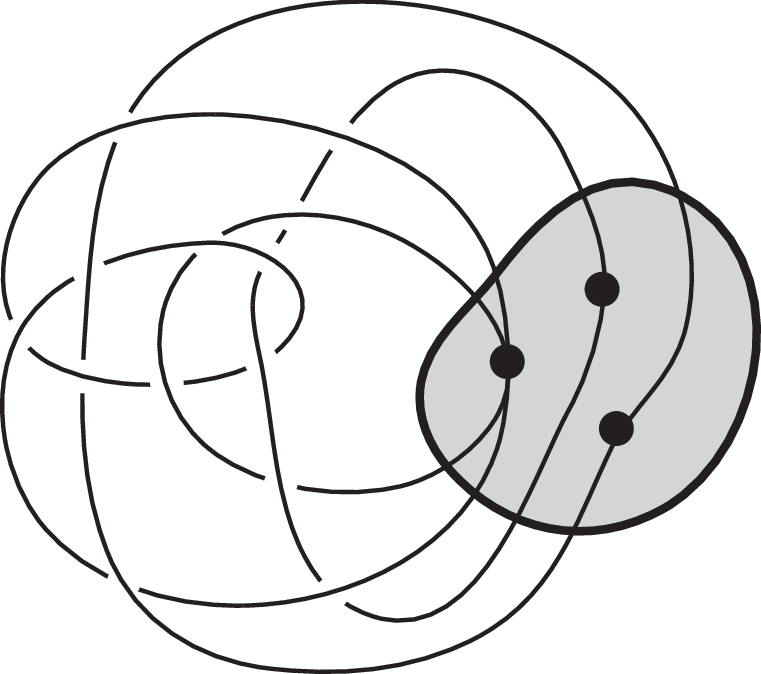}
\put(110,16){$b$} \put(4,90){$\Gamma$} 
\end{overpic}}
\quad\rightarrow\quad
\raisebox{-50 pt}{\begin{overpic}[width=120pt]{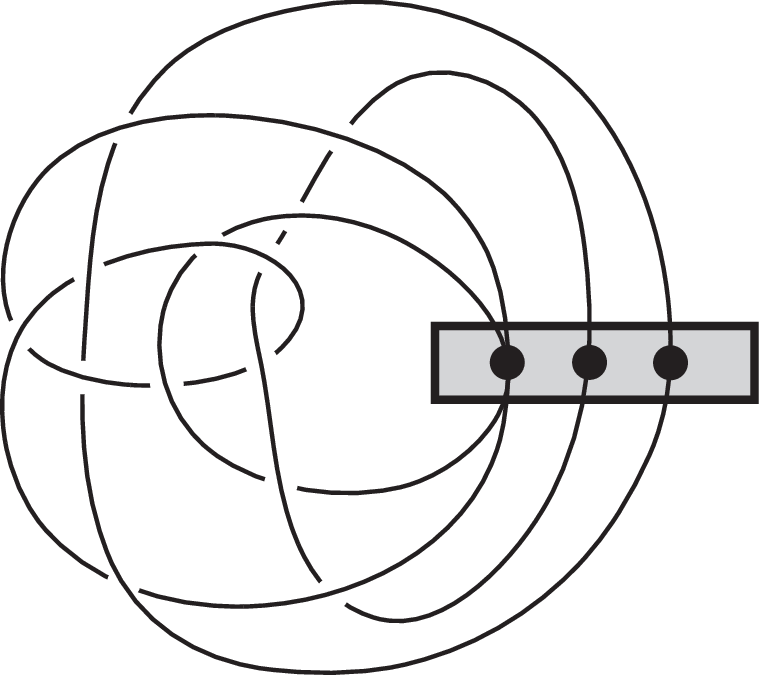}
\put(111,32){$b$} \put(4,90){$\Gamma$} 
\end{overpic}}
\quad\rightarrow\quad
\raisebox{-50 pt}{\begin{overpic}[width=120pt]{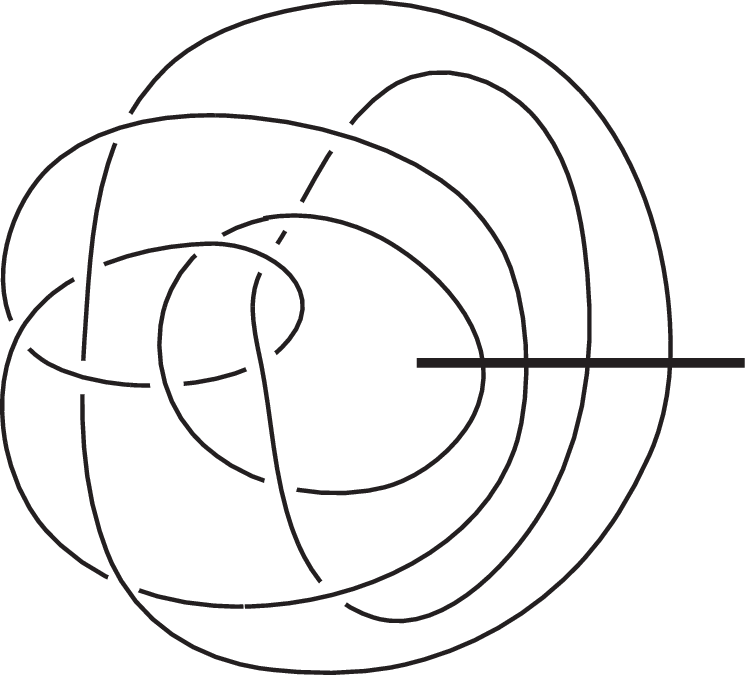}
\put(111,55){$D$} \put(4,90){$L$} 
\end{overpic}}
$$ 
\vspace{0.2cm}
\caption{A b-based spatial B-graph $(\Gamma,b)$ and a corresponding d-based link $(L,D)$.}\label{b-disc}
\end{figure}
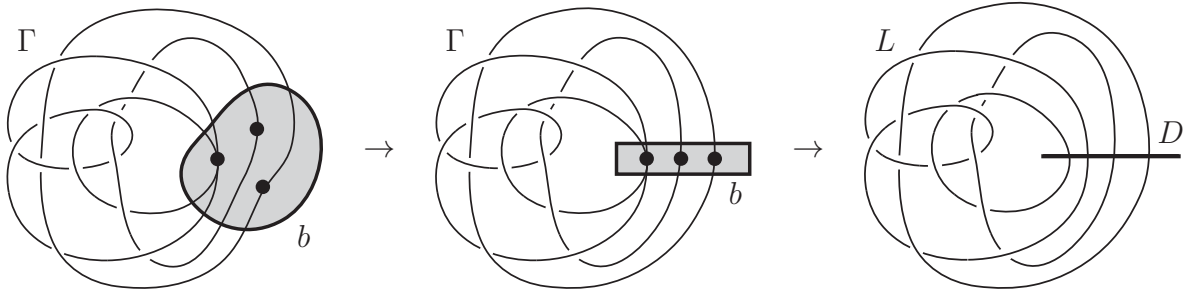

Then we have a Markov-type theorem for the component-homotopy classes of spatial B-graphs. 

\begin{theorem}\label{Markov-type-for-Bgraph}
Let $\sigma$ and $\sigma'$ be in $\mathcal{CH}({\bm l})$. Then the G-closures of $\sigma$ and $\sigma'$ are component-homotopic if and only if there is an element $s\in\mathcal{S}_{G}({\bm l})$ such that $s\cdot\sigma=\sigma'$. 
\end{theorem}

\begin{proof}
This theorem is proved as an analogue of Theorem \ref{Markov-type-for-CL} (and Theorem 2.9 in \cite{HL}) by replacing the d-bases for colored links with the b-bases for spatial B-graphs. 
\end{proof}

When we take the G-cap of $\Sigma\in \mathcal{CH}(\overline{\bm l}{\bm l})$, the sub-string links $\sigma_{\overline{i}}$ and $\sigma_i$ become the same component. 
Let $RCF^*(\overline{\bm l}{\bm l})$ be the quotient group of $RCF(\overline{\bm l}{\bm l})$ with relations such that conjugations of $x_{ij}$ commute with those of $\overline{x}_{ik}$. 
Let $\mathcal{CH}^*(\overline{\bm l}{\bm l})$ be the quotient group of $\mathcal{CH}(\overline{\bm l}{\bm l})$ with crossing changes between arcs of $(i,j)$-th component and those of ($\overline{i},k$)-th component. The component-homotopy class of the G-cap $\Sigma\cdot G_0$ of $\Sigma\in \mathcal{CH}(\overline{\bm l}{\bm l})$ is determined by the projection of $\Sigma$ to $\mathcal{CH}^*(\overline{\bm l}{\bm l})$. 
Then, corresponding to Theorem \ref{CSLclassify}, we have a split exact sequence
\begin{equation}
1\to RCF^*(\overline{\bm l_{m-1}}{\bm l}_{m-1})^{2l_m}\overset{}{\to} \mathcal{CH}^*(\overline{\bm l}{\bm l})\to \mathcal{CH}^*(\overline{\bm l_{m-1}}{\bm l}_{m-1})\to 1, \label{CSL2-split}
\end{equation}
where $\mathcal{CH}^*(\overline{\bm l}{\bm l})\to \mathcal{CH}^*(\overline{\bm l_{m-1}}{\bm l}_{m-1})$ is defined by omitting the sub-string links $\sigma_{m}$ and $\sigma_{\overline{m}}$, the split homomorphism $\mathcal{CH}^*(\overline{\bm l_{m-1}}{\bm l}_{m-1})\to \mathcal{CH}^*(\overline{\bm l}{\bm l})$ adds the two trivial sub-string links with $l_m$ components as $\sigma_{m}$ and $\sigma_{\overline{m}}$ and $RCF^*(\overline{\bm l_{m-1}}{\bm l}_{m-1})^{2l_m}\overset{}{\to} \mathcal{CH}^*(\overline{\bm l}{\bm l})$ adds $2l_m$ strings $\sigma_{m}$ and $\sigma_{\overline{m}}$ which are determined by an element of $RCF^*(\overline{\bm l_{m-1}}{\bm l}_{m-1})^{2l_m}$ to the trivial string link with the component decomposition $\overline{\bm l_{m-1}}{\bm l}_{m-1}$. 
\par
We also define a group $\mathcal{S}^*_G({\bm l})$ as  $\mathcal{S}_G({\bm l})$ modulo crossing changes between arcs of $(i,j)$-th component and those of ($\overline{i},k$)-th component. Consider the following split short exact sequence which is induced by (\ref{CSL2-split}).  
$$1\to\mathcal{K}_G\overset{\Psi}{\to} \mathcal{S}^*_G({\bm l}) \to \mathcal{S}^*_G({\bm l}_{m-1}) \to 1.$$
Note that $\mathcal{K}_G$ is a subgroup of $RCF^*(\overline{\bm l_{m-1}}{\bm l}_{m-1})^{2l_m}
=\{RCF^*(\overline{\bm l_{m-1}}{\bm l}_{m-1})^2\}^{l_m}$. 
Let $\mathcal{K}_0$ be the subgroup of $\mathcal{K}_G$ generated by 
the following elements
\begin{eqnarray*}
& & (x_{ij}, x_{ij})^{l_m},\quad (\overline{x}_{ij}, \overline{x}_{ij})^{l_m},\quad 
((1,1),\cdots,\underset{t}{(\overline{X}_i^{-1}X_i,1)},\cdots,(1,1)) \\ 
\mbox{and\!\!\!\!}& &((1,1),\cdots,\underset{t}{(1,\overline{X}_i^{-1}X_i)},\cdots,(1,1)) 
\end{eqnarray*}
$(1\leq i\leq m-1, 1\leq j\leq l_i, 1\leq t\leq l_m)$, where these elements are mapped to
\begin{eqnarray*}
\Psi((x_{ij}, x_{ij})^{l_m})=\prod_{t=1}^{l_m}(x_{ij},x_{ij})_{mt}, \,\,
\Psi((\overline{x}_{ij}, \overline{x}_{ij})^{l_m})=\prod_{t=1}^{l_m}(\overline{x}_{ij}, \overline{x}_{ij})_{mt},\\
\Psi(((1,1),\cdots,\underset{t}{(\overline{X}_i^{-1}X_i,1)},\cdots,(1,1)))
=(\overline{X}_i^{-1}X_i,1)_{mt}, \\
\Psi(((1,1),\cdots,\underset{t}{(1,\overline{X}_i^{-1}X_i)},\cdots,(1,1)))
=(1,\overline{X}_i^{-1}X_i)_{mt}.
\end{eqnarray*}


\begin{remark} \label{trivial-condition}
In the case of $l_m=1$,  
for $(f,g)\in RCF^*(\overline{\bm l_{m-1}}{\bm l}_{m-1})^{2}$, $(f,g)\in \mathcal{K}_G$ if and only if $gf^{-1}$ is a product of conjugations of $\overline{X}_i^{-1}X_i$ and their inverses. 
\end{remark}

\begin{proposition} $\mathcal{K}_0=\mathcal{K}_G.$
\end{proposition}

\begin{proof}
First we prove the case $l_m=1$. In this case, we specially denote $\mathcal{K}_0$ by $\mathcal{K}_0^{(1)}$. 
Then $\mathcal{K}_G\subset RCF^*(\overline{\bm l_{m-1}}{\bm l}_{m-1})^{2}$ and $\mathcal{K}_0^{(1)}$ is generated by 
$$(x_{ij}, x_{ij}),\,\, (\overline{x}_{ij}, \overline{x}_{ij}),\,\, (\overline{X}_i^{-1}X_i,1),\,\, (1,\overline{X}_i^{-1}X_i) \qquad(1\leq i\leq m-1, 1\leq j\leq l_i).$$
Then we have $(1, h\overline{X}_i^{-1}X_i h^{-1})=(h,h)(1, \overline{X}_i^{-1}X_i)(h^{-1},h^{-1})\in \mathcal{K}_0^{(1)}$ for $h\in RCF^*(\overline{\bm l_{m-1}}{\bm l}_{m-1})$. 
\par
From Remark \ref{trivial-condition}, for $(f,g)\in \mathcal{K}_G$, $gf^{-1}$ is a product of conjugations of $\overline{X}_i^{-1}X_i$ and 
$(f,g)=(1,gf^{-1})(f,f)\in \mathcal{K}_0^{(1)}$. Therefore $\mathcal{K}_0^{(1)}\supset\mathcal{K}_G$ and $\mathcal{K}_0^{(1)}=\mathcal{K}_G$. 
\par
We prove the case $l_m>1$. 
For an element
$\sigma=((f_1,g_1), (f_2,g_2),\dots, (f_{m_l},g_{m_l}))\in \mathcal{K}_G$, 
the G-cap $\Psi(\sigma)\cdot G_0$ of the image $\Psi(\sigma)$ is component-homotopic to $G_0\subset D^2\times I$. Let $G_0^{(m-1)}\subset D^2\times I$ be the sub-spatial graph of $G_0$ obtained by omitting the $m$-th component of $G_0$. 
Let $\lambda$ be the element of the fundamental group $\pi_1(D^2\times I \setminus G_0^{(m-1)})$ which is the loop 
consisting of the segment connecting the basepoint $b$ and $p_{m1}$, the one connecting $b$ and $p_{\overline{m1}}$
and $m1$-th and $\overline{m1}$-th edges of $G_0$ with the same orientation to the edges. 
Since $G_0$ is isotopic to a plane embedding of star graphs, $\lambda=g_1f^{-1}_1$ is trivial in $\pi_1(D^2\times I \setminus G_0^{(m-1)})$. 
Thus $g_1f_1^{-1}$ is a product of conjugations of $\overline{X}_i^{-1}X_i$ and their inverses and $(f_1,g_1)\in \mathcal{K}_0^{(1)}$. 
\par
Let $\lambda_{1t}$ (resp. $\lambda_{\overline{1t}}$) be the element of $\pi_1(D^2\times I \setminus G_0^{(m-1)})$ which is the loop 
consisting of the segment connecting the basepoint $b$ and $p_{m1}$ (resp. $p_{\overline{m1}}$), the one connecting $b$ and $p_{\overline{mt}}$ (resp. $p_{\overline{mt}}$)
and $m1$-th and $mt$-th edges (resp. $\overline{m1}$-th and $\overline{mt}$-th) of $G_0$ with the same orientation to the $mt$-th (resp. $\overline{mt}$-th) edge. The loop $\lambda_{1t}$ is trivial as an element of $\pi_1(D^2\times I \setminus G_0^{(m-1)})$. 
Thus there exists a product $k_t$ of conjugations of $\overline{X}_i^{-1}X_i$ and their inverses such that $g_{t}g_{1}^{-1}=k_t$ for every $t$. 
Similarly, since $\lambda_{\overline{1t}}$ is trivial in $\pi_1(D^2\times I \setminus G_0^{(m-1)})$, there is a product $k_{\overline{t}}$ of conjugations of $\overline{X}_i^{-1}X_i$ and their inverses such that $f_{t}f_{1}^{-1}=k_{\overline{t}}$ for every $t$. Then
$$\sigma=((f_1,g_1), (k_{\overline{2}}f_1,k_{2}g_1),\dots, (k_{\overline{m-1}}f_1,k_{m-1}g_1))\in \mathcal{K}_G.$$
Similar to the discussion in the case $l_m=1$, 
\begin{eqnarray*}
& &((h,h),\cdots,(h,h))
\cdot ((1,1),\cdots,\underset{t}{(1,\overline{X}_i^{-1}X_i)},
\cdots,(1,1))
\cdot
((h^{-1},h^{-1}),\cdots,(h^{-1},h^{-1}))\\
&=&((1,1),\cdots,\underset{t}{(1,h \overline{X}_i^{-1}X_ih^{-1})},
\cdots,(1,1))\in \mathcal{K}_0. 
\end{eqnarray*}
for $h\in RCF^*(\overline{\bm l_{m-1}}{\bm l}_{m-1})$. 
Similarly, $((1,1),\cdots,\underset{t}{(h \overline{X}_i^{-1}X_ih^{-1},1)},
\cdots,(1,1))\in \mathcal{K}_0$. Then
$$\sigma=((1,1),(k_{\overline{2}},k_2),\dots, (k_{\overline{m-1}},k_{m-1}))
\cdot((1,g_1f_1^{-1}),\cdots,(1,g_1f_1^{-1}))
\cdot((f_1,f_1),\cdots,(f_1,f_1))$$
is in $\mathcal{K}_0$ and $\mathcal{K}_0=\mathcal{K}_G$. 
\end{proof}

From the decomposition $\mathcal{H}({\bm l})=\mathcal{H}({\bm l}_{m-1})\ltimes RCF({\bm l}_{m-1})^{l_m}$ in Theorem \ref{CSLclassify}, $\sigma\in \mathcal{H}({\bm l})$ is presented as $\sigma=\theta{\bm g}$, where $\theta\in \mathcal{H}({\bm l}_{m-1})$ and ${\bm g}=(g_1, \dots, g_{l_m})\in RCF({\bm l}_{m-1})^{l_m}$. Then the actions of $(\overline{x}_{ij}, \overline{x}_{ij})^{l_m}$ and $(x_{ij}, x_{ij})^{l_m}$ are 
$$(\overline{x}_{ij}, \overline{x}_{ij})^{l_m}\cdot\theta{\bm g}=\theta (g_1^{x_{ij}}, \dots, g_{l_m}^{x_{ij}}), \quad
(x_{ij}, x_{ij})^{l_m}\cdot\theta{\bm g}=\theta (g_1^{\theta^{-1}x_{ij}\theta}, \dots, g_{l_m}^{\theta^{-1}x_{ij}\theta}).$$ 
So the actions of $(x_{ij}, x_{ij})^{l_m}$ are generated by those of $(\overline{x}_{ij}, \overline{x}_{ij})^{l_m}$. 
\par
The decomposition (\ref{CLdecomposition}) is done with respect to the $m$-colored components. We can do a similar decomposition with respect to the $i$-colored component and have a subgroup of $\mathcal{S}^*_G({\bm l})$ which corresponds to $\mathcal{K}_G$. The union of generators of such subgroups generates $\mathcal{S}^*_G({\bm l})$. So $\mathcal{S}^*_G({\bm l})$ is generated by 
$\displaystyle (\overline{x}_{st}, \overline{x}_{st})_i:=\prod^{l_i}_{j=1}(\overline{x}_{st}, \overline{x}_{st})_{ij}$ , $(\overline{X}_s^{-1}X_s,1)_{ij}$ and 
$(\overline{X}_s^{-1}X_s,1)_{ij}$ \,\,
($1\leq i, s \leq m, i\neq s, 1\leq i\leq l_j, 1\leq t\leq l_s$)
 (see Figure \ref{Ggenerator}). 
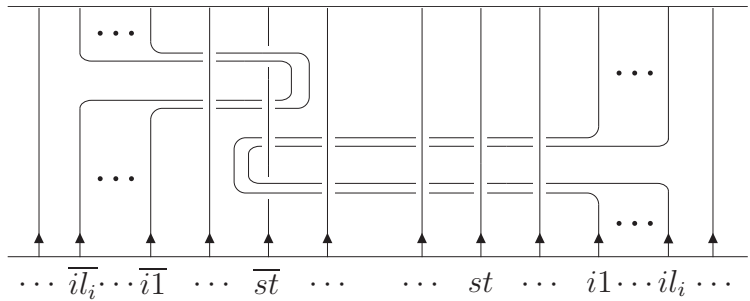
\begin{figure}[h]
$$
\raisebox{-20 pt}{\begin{overpic}[width=280pt]{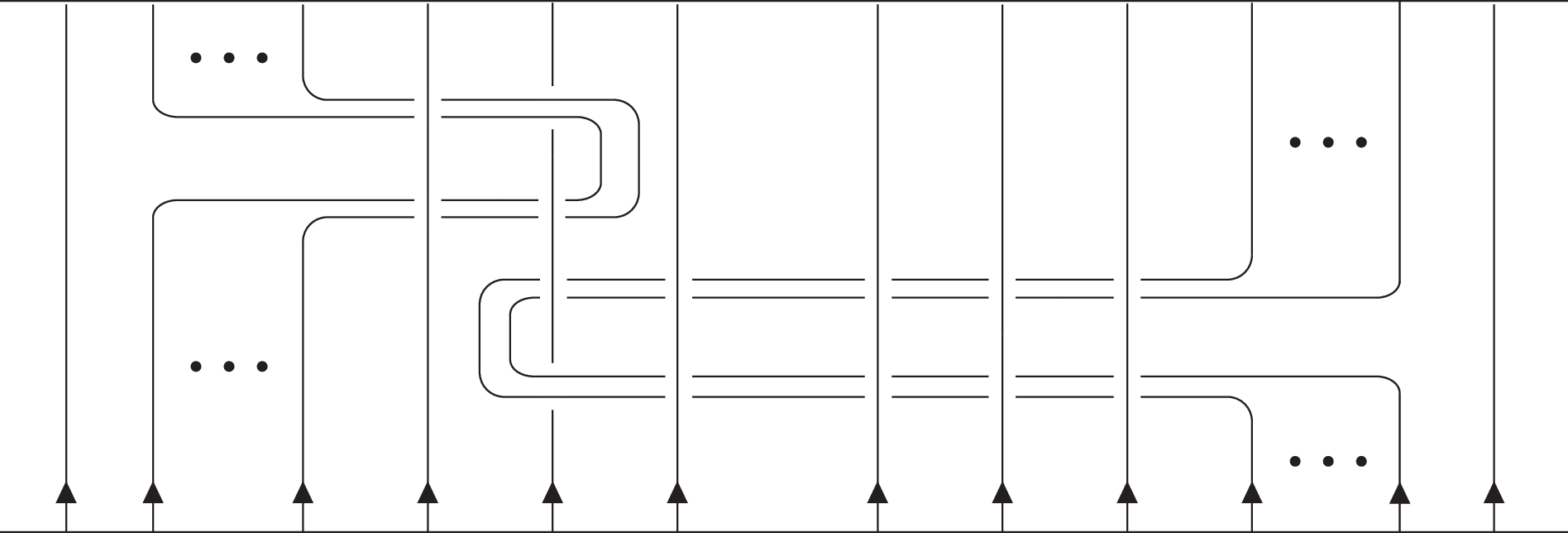}
\put(4,-13){$\cdots$} 
\put(23,-14){$\overline{il_i}$} \put(34,-13){$\cdots$} \put(50,-14){$\overline{i1}$} \put(70,-13){$\cdots$}
\put(93,-14){$\overline{st}$} \put(114,-13){$\cdots$}
\put(149,-13){$\cdots$} \put(174,-13){$st$} \put(194,-13){$\cdots$} \put(219,-13){$i1$} \put(230,-13){$\cdots$}
\put(246,-13){$il_i$} \put(260,-13){$\cdots$}
\end{overpic}}
$$ 
\vspace{0.2cm}
\caption{Generators $(\overline{x}_{st}, \overline{x}_{st})_i$ of $\mathcal{S}^*_G({\bm l})$.}\label{Ggenerator}
\end{figure}

Moreover, we see that 
$$(\overline{X}_s^{-1}X_s,1)_{ij} = (\overline{x}_{ij}^{-1}, \overline{x}_{ij}^{-1})_s \,\,\mbox{ and }\,\, (1,\overline{X}_s^{-1}X_s)_{ij} = (x_{ij}, x_{ij})_s$$
in $\mathcal{S}^*_G({\bm l})$ and have the following lemma. 

\begin{lemma}
$\mathcal{S}^{*}_{G}({\bm l})$ is generated by the elements $\displaystyle (\overline{x}_{st}, \overline{x}_{st})_i$. 
\end{lemma}

Then we have the Markov-type theorem for the colored string links and the spatial B-graphs as follows. 

\begin{theorem}
Let $\sigma$, $\sigma'\in \mathcal{CH}({\bm l})$. Then the G-closures of $\sigma$ and $\sigma'$ are component-homotopic as spatial B-graphs if and only if there is a sequence $\sigma=\sigma_0, \sigma_1,\dots,\sigma_n=\sigma'$ of elements of $\mathcal{CH}({\bm l})$ such that $\sigma_{j+1}=(\overline{x}_{st}, \overline{x}_{st})_i\cdot \sigma_j$ for some $i$, $s$ and $t$.  
\end{theorem}

\subsection{The Habegger-Lin Algorithm for colored links and spatial graphs}
Habegger and Lin \cite{HL} showed that there is an algorithm which determine whether given two links are link-homotopic or not. 
The algorithm uses the actions of partial conjugations for string links. For the 4- and 5-component case, the actions are calculated in \cite{KM2, Gra, KM3} explicitly and we can run the algorithm. 
\par 
The algorithm is available for the CL-homotopy classes of colored links (resp. the component-homotopy classes of spatial graphs) since the properties of $\mathcal{H}(n)$ in Lemmas 3.2, 3.3 and 3.4 in \cite{HL} also hold for $\mathcal{CH}({\bm l})$ and $\mathcal{S}_{CL}$ (resp. $\mathcal{S}_G$). 
Thus there is an algorithm which determine whether given two colored links (resp. two B-graphs) are CL-homotopic (resp. component-homotopic) or not.


\begin{thebibliography}{99}
 \bibitem{Gra} E.~Graff, {\it On braids and links up to link-homotopy}, J. Math. Soc. Japan Advance Publication 1 -- 35, May, 2023. https://doi.org/10.2969/jmsj/90449044
\bibitem{F} T.~Fleming, {\it Milnor invariants for spatial graphs}, Topol. Appl. \textbf{155} (2008) 1297--1305.
\bibitem{FT} M.~Freedman and P.~Teichner, {\it 4-manifold topology. I. Subexponential groups}, Invent. Math. \textbf{122} (1995), no.3, 509--529.
4-manifold topology I. Subexponential groups
\bibitem{HL} N.~Habegger and X.-S.~Lin,{\it The classification of links up to link-homotopy}, J. Amer. Math. Soc. 3:2 (1990), 389--419.
\bibitem{Ha}  K.~Habiro, {\it Claspers and finite type invariants of links}, Geom. Topol. \textbf{4} (2000), 1--83.
\bibitem{Hu} J.~R.~Hughes, {\it Partial conjugations suffice}, Topology Appl. \textbf{148} (2005),  55--62.
\bibitem{KM2} Y.~Kotorii and A.~Mizusawa, {\it Link-homotopy classes of 4-component links, claspers and the Habegger-Lin algorithm}, J. Knot Theory Ramifications, \textbf{32} 2023, no.6, 2350045, 26 pp.
\bibitem{KM3} Y.~Kotorii and A.~Mizusawa, {\it Clasper presentations of Habegger-Lin's action on string links}, arXiv:2212.14502.
\bibitem{Le2} J. P. Levine, {\it An approach to homotopy classification of links}, Trans. Amer. Math. Soc. \textbf{306} (1988), 361--387.
\bibitem{MY2} J-B.~Meilhan, Y. Yasuhara, {\it Milnor invariants and the HOMFLYPT polynomial}, Geom. Topol. \textbf{16} (2012), no.2, 889--917.
\bibitem{Mil} J.~Milnor, \textit{Link groups}, Annals of Mathematics (2), \textbf{59} (1954), p177--195.
\bibitem{Mil2} J.~Milnor, {\it Isotopy of links}, Algebraic geometry and topology, A symposium in honor of S. Lefschetz, pp. 280--306, Princeton University Press, Princeton, N. J., 1957.
\bibitem{NM} V.~M.~Nezhinskii and Yu.~V.~Maslova, \textit{Homotopy classification of three-component singular links of graphs} (Russian), 
Uspekhi Mat. Nauk \textbf{63} (2008), no. 5(383), 195--196; translation in
Russian Math. Surveys \textbf{63} (2008), no. 5, 976--977.
\bibitem{Ni} R.~Nikkuni, {\it Edge-homotopy classification of spatial complete graphs on four vertices}, J. Knot Theory Ramifications \textbf{13} (2004), no.6, 763--777.
\bibitem{Sta} J. Stallings, {\it Homology and central series of groups}, J. Algebra \textbf{20} (1965), 170--181.
\bibitem{T} K.~Taniyama, {\it Cobordism, homotopy and homology of graphs in $\mathbb{R}^3$}, Topology \textbf{33} (1994), no.3, 509--523.
\bibitem{Y} A.~Yasuhara, \textit{Self delta-equivalence for links whose Milnor's isotopy invariants vanish}, Trans. Amer. Math. Soc. \textbf{361} (2009), 4721--4749.
\end{thebibliography}
\end{document}